\documentclass[lettersize]{amsart}

\usepackage{enumitem} 
\usepackage{amsmath,amsthm,amssymb,amscd,amsthm,url,mathtools} 
\usepackage{thmtools}
\usepackage[all]{xy}
\usepackage{braket, dsfont}
\usepackage{graphicx} 
\usepackage{xspace} 
\usepackage{calc} 

\usepackage[margin=1.3in]{geometry}
\usepackage[usenames,dvipsnames]{xcolor}
\usepackage{tikz}
\usepackage{pgfplots}

\usepackage{bm}
\usepackage{microtype} 

\newcommand{\mute}[1] {}

\newcommand{\mc}[1]{\mathcal{#1}}

\def \Z{\mathbb{Z}}
\def \R{\mathbb{R}}
\def \C{\mathbb{C}}
\def \HH{\mathbb{H}}

\def \del{\partial}

\def \vphi{\varphi}
\def \eps{\varepsilon}
\def \ssm{\smallsetminus}

\newtheorem*{Theorem*}{Theorem}
\newtheorem*{Lemma*}{Lemma}
\newtheorem*{Proposition*}{Proposition}
\newtheorem{Theorem}{Theorem}
\newtheorem{Lemma}[Theorem]{Lemma}
\newtheorem{Proposition}[Theorem]{Proposition}

\newtheorem*{Claim}{Claim}
\newtheorem{Corollary}[Theorem]{Corollary}

\theoremstyle{definition}
\newtheorem{Definition}[Theorem]{Definition}

\newtheorem{Fact}[Theorem]{Fact}

\newtheorem{Example}[Theorem]{Example}

\theoremstyle{remark}
\newtheorem{Remark}[Theorem]{Remark}

\DeclareMathOperator{\Env}{Env}
\DeclareMathOperator{\In}{In}
\DeclareMathOperator{\Outenv}{Out}
\DeclareMathOperator{\Span}{Span}

\newcommand{\te}{\stackrel{{}_\ast}{\asymp}}
\newcommand{\tg}{\stackrel{{}_\ast}{\succ}}
\newcommand{\tl}{\stackrel{{}_\ast}{\prec}}
\newcommand{\pe}{\stackrel{{}_+}{\asymp}}

\newcommand{\pl}{\stackrel{{}_+}{\prec}}

\newcommand{\PML}{\ensuremath{\mathcal{PML}}\xspace}

\begin{document}

\title{Geodesic Envelopes in the Thurston Metric 
on Teichm\"{u}ller space}

\author{Assaf Bar-Natan}

\address{Department of Mathematics, Brandeis University, Waltham, MA}

\email{assaf.bar.natan@gmail.com}

\date{\today}

\begin{abstract}
  The Thurston metric on Teichm\"{u}ller space, first introduced by W.
  P. Thurston is an asymmetric metric on Teichm\"{u}ller space defined
  by $d_{Th}(X,Y) = \frac12 \log\sup_{\alpha}
  \frac{l_{\alpha}(Y)}{l_{\alpha}(X)}$. This metric is geodesic, but
  geodesics are far from unique. In this thesis, we show that in the
  once-punctured torus, and in the four-times punctured sphere,
  geodesics stay a uniformly-bounded distance from each other.  In other
  words, we show that the \textbf{width} of the \textbf{geodesic
  envelope}, $E(X,Y)$ between any pair of points $X,Y \in \mc{T}(S)$
  (where $S = S_{1,1}$ or $S = S_{0,4}$) is bounded uniformly.  To do
  this, we first identify extremal geodesics in $Env(X,Y)$, and show
  that these correspond to \textbf{stretch vectors}, proving
  a conjecture from \cite{HOP21}. We then compute Fenchel-Nielsen
  twisting along these paths, and use these computations, along with
  estimates on earthquake path lengths, to prove the main theorem.
\end{abstract}

\maketitle
\vspace{-0.25cm}

\begin{figure}[ht]
  \begin{center}
    \includegraphics[width=0.3\textwidth]{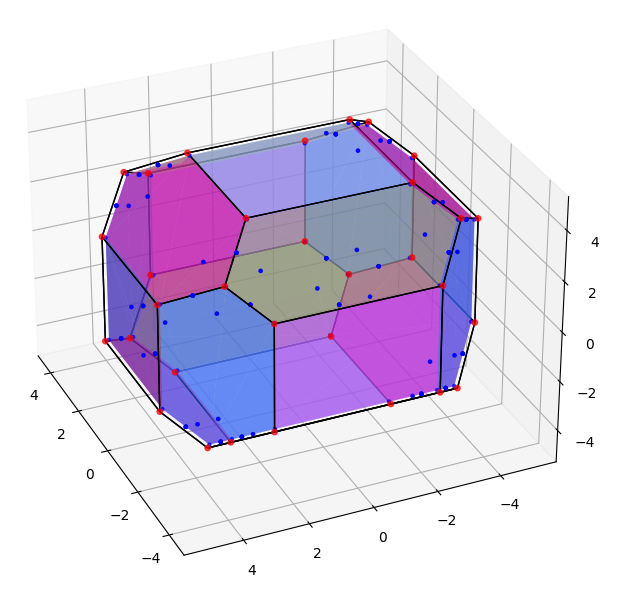}
  \end{center}
  \caption{The unit tangent sphere is the convex hull of stretch 
  vectors (in red). Combinatorially, this is a chamfered cube.}
\end{figure}

\section{Introduction} \label{sec:intro}

Let $S$ be a closed orientable surface with negative Euler
characteristic, and let $\mc{T}(S)$ be its Teichm\"{u}ller space. In
this paper, we study the infinitesimal geometry of the Thurston
metric, $d_{Th}$ on $\mc{T}(S)$, answering a question of Huang,
Ohshika and Papadopoulos \cite{HOP21} in the affirmative. We then
focus on two low-complexity cases, where the infinitesimal results are
used to show that Thurston geodesics in these Teichm\"{u}ller spaces
have uniformly bounded width. 

In Thurston's paper on the Thurston metric \cite{Thu86}, he proves
that $(\mc{T}(S), d_{Th})$ is a geodesic metric space by constructing
a special family of paths called \textit{stretch paths}. For any $X
\in \mc{T}(S)$, we define the \textit{stretch vectors at $X$}, $SV_X$
to be the set of $1$-jets of stretch paths starting at $X$. For 
each $X$ we can also look at the unit tangent sphere, $\bm{S}_X$ 
which strictly contains $SV_X$. We prove:

\begin{Theorem} \label{thm:stretch-vectors-extremal}
  $\bm{S}_X$ is the convex hull of $SV_X$. Moreover, $SV_X$ is
  precisely the set of extreme points in $\bm{S}_X$.
\end{Theorem}

Geodesics in the Thurston metric are far from unique \cite{DLRT20,
LRT14}. We define the \textit{geodesic envelope} between $X$ and $Y$
to be the collection of points along geodesics between $X$ and $Y$,
and in general this envelope has unbounded \textit{width}, 
denoted by $w(X,Y)$ (ie, 
two geodesics between $X$ and $Y$ can grow arbitrarily far 
apart from each other).

In Section \ref{sec:bounded-width}, we begin by computing 
estimates for the Thurston distance 
between a point $X$ and its earthquake 
along a simple closed curve $\alpha$, in terms of 
$l_{\alpha}(X)$. We prove:

\begin{Proposition} \label{prop:earthquake-bound}
  Let $\alpha$ be a simple closed curve on $S$, and let 
  $X \in \mc{T}(S)$. Then there exists some uniform constant 
  $C$ such that 
  \begin{align*}
    d_{Th}(X,Eq_{\alpha,t}(X)) \le \log(e^{l_{\alpha}/2}t) + C
  \end{align*}
\end{Proposition}
We then use this result, combined with the 
computations in Section \ref{sec:infinitesimal-envelope} 
to show that:

\begin{Theorem} \label{thm:bounded-envelope}
  Let $S$ be the once-punctured torus or the four-times punctured
  sphere. There exists some $B>0$ such that for any $X,Y \in
  \mc{T}(S)$, $w(X,Y) < B$.
\end{Theorem}

\noindent \textbf{Organization.} We begin by going through 
some background on the Thurston metric on Teichm\"{u}ller 
space, laminations and shearing co-ordinates, and Fenchel-Nielsen 
twisting co-ordinates. Section \ref{sec:infinitesimal-envelope} 
is dedicated to proving Theorem \ref{thm:stretch-vectors-extremal}, 
and in Section \ref{sec:bounded-width}, we 
prove Theorem \ref{thm:bounded-envelope} using some computations 
from subsection \ref{subsec:twist-compute}.

\noindent \textbf{Acknowledgements.} This paper is based on the
author's PhD thesis. I would like to thank my advisor, Kasra Rafi, who
has given me endless support and motivation throughout the writing of
this work. 

I would also like to thank my committee members: Alexander Nabutovsky
and Giulio Tiozzo. Many thanks to Jing Tao, Yair Minsky for helpful
discussion and to Francis Bonahon, who helped me with the proof of
Lemma \ref{lem:finite-stretch-vectors-span-hull}, and to Yi Huang, who
helped me clarify the proof of Theorem
\ref{thm:stretch-vectors-extremal}. Special thanks to David Dumas, for
reading and commenting on earlier drafts of this paper.

Finally, I would like to acknowledge support from the Univeristy of
Toronto mathematics department, the Faculty of Arts and Sciences, the
Natural Sciences and Engineering Research Council of Canada.

\section{Background} \label{sec:background}


\subsection{Teichmuller space \& Thurston Metric}

Throughout this paper, we let $S$ be an orientable surface with no
boundary components, and negative Euler characteristic.  In
particular, this means that $S$ can be endowed with a hyperbolic
metric. 

\begin{Definition}
  A \textbf{marking} on $S$ is a homeomorphism 
  $f:S \to X$, where $X$ is a surface endowed with a hyperbolic metric.
\end{Definition}

Two markings $f:S\to X$ and $g:S \to Y$ are equivalent 
if $fg^{-1}$ is homotopic to an isometry.

\begin{Definition}
  We define the \textbf{Teichm\"{u}ller space} of 
  $S$, denoted by $\mc{T}(S)$, to be the space 
  of all markings on $S$ up to equivalence.
\end{Definition}

We think of points in Teichm\"{u}ller space 
as a pair consisting of a metric space $X$, together with a marking 
map $\vphi$.

If $\gamma$ is a curve, arc, or arc segment on $S$, then we define
$l_{\gamma}:\mc{T}(S) \to \R_{\ge 0}$ by sending $X$ to
$l_{\gamma}(X)$, the \textbf{length} of the geodesic representative of
$\gamma$ on $X$ relative to its endpoints.

\begin{Definition}
We define the \textbf{Thurston metric} $d_{Th}:\mc{T}(S)\times
\mc{T}(S)\to \R_{\ge0}$ 
by:
\begin{align*}
  d_{Th}(X,Y) = \sup_{\alpha} \log\left(\frac{l_{\alpha}(Y)}{l_{\alpha}(X)}\right)
\end{align*}
where the supremum ranges over all simple closed curves $\alpha$ 
contained in $S$.
\end{Definition}

\begin{Lemma} \label{lem:dTh-defined}
  For any $X,Y\in \mc{T}(S)$, 
  $d_{Th}(X,Y)$ is equal to:
  \begin{itemize}
    \item $d_{Th}(X,Y) = \sup_{\alpha} \log\left(\frac{l_{\alpha}(Y)}{l_{\alpha}(X)}\right)$ (as above)
    \item $L(X,Y) = \inf_{f \sim id} \log(L_f)$, where $f:X\to Y$ is
      a Lipschitz map homotopic to the identity, and $L_f$ is the
      Lipschitz constant of $f$.
    \item $D(X,Y) = \inf_{f \sim id} \sup_{p\in X} \log(\|D
      f_p\|)$, where $f:X\to Y$ is a homeomorphism that 
      is once differentiable almost everywhere.
  \end{itemize}
\end{Lemma}

\begin{proof}
  In \cite{Thu86}, Thurston shows that 
  $L(X,Y) = d_{Th}(X,Y)$. It suffices to show that 
  $d_{Th}(X,Y) \le D(X,Y) \le L(X,Y)$. The first 
  inequality follows because $\frac{l_{\alpha}(Y)}{l_{\alpha}(X)}$ 
  is always bounded above by $\sup_p \|Df_p\|$. The 
  latter inequality is more subtle. Trivially, 
  if in the definition of $L(X,Y)$, 
  $f$ is taken to be differentiable, then the inequality 
  follows by the fact that Lipschitz constants 
  give upper bounds for derivatives. In fact, 
  Thurston \cite{Thu86} explicitly 
  constructs a map $f$ that realizes 
  the infimum in $L(X,Y)$, and this map is a 
  homeomorphism that is differentiable almost 
  everywhere. Thus, the inequality holds.
\end{proof}

\begin{Claim}
  For any $X,Y,Z \in \mc{T}(S)$, the Thurston metric satisfies:
  \begin{itemize}
    \item $d_{Th}(X,Y) \ge 0$
    \item $d_{Th}(X,Y) = 0$ if and only if $X = Y$
    \item $d_{Th}(X,Z) \le d_{Th}(X,Y) + d_{Th}(Y,Z)$
  \end{itemize}
\end{Claim}
\begin{proof}
  We use the third characterization of $d_{Th}$ in 
  Theorem \ref{lem:dTh-defined}. 
  \begin{itemize}
    \item If $d_{Th}(X,Y) < 0$, then that would imply the 
      existence of a map $f:X\to Y$ such that $\|Df_p\| < 1$ 
      for all $p$. In particular, this means that 
      $f$ is not area-preserving, which cannot happen, 
      as $\mathrm{Area}(X) = \mathrm{Area}(Y) = -2\pi \chi(S)$.
    \item If $d_{Th}(X,Y) = 0$, then there would be 
      a homeomorphism $f:X\to Y$ homotopic to the identity 
      with $\|Df_p\| = 1$ for all $p$. In particular, 
      $f$ must be an isometry, and $X$ and $Y$ are equivalent.
    \item This follows immediately from the chain rule.
  \end{itemize}
\end{proof}

The functions $d_{Th} = L = D$ define an 
asymmetric complete geodesic metric on $\mc{T}(S)$ \cite{Thu86}. By 
this, we mean that:
\begin{Theorem} 
  For any $X,Y,Z\in \mc{T}(S)$, 
  \begin{itemize}
    \item $d_{Th}(X,Y) \ge 0$ for all $X,Y$, 
      with equality if and only if $X=Y$
    \item $d_{Th}(X,Y) \le d_{Th}(X,Z) + d_{Th}(Z,Y)$
    \item There exists a path $\gamma:[0,d_{Th}(X,Y)]\to \mc{T}(S)$ 
      such that $\gamma(0)=X$, $\gamma(d_{Th}(X,Y)) = Y$, and 
      for any $s\le t$, $d_{Th}(\gamma(s),\gamma(t)) = t-s$.
  \end{itemize}
\end{Theorem}

For any simple closed curve $c$, and any $v\in T_X\mc{T}(S)$, 
we define $D_v\log l_c = \frac{d}{dt}|_{t=0} \log(l_c(\alpha(t)))$, 
where $\alpha(t)$ is some germ whose $1$-jet is equal to $v$. This 
family of linear functionals induces the Thurston norm on
$T_X\mc{T}(S)$:
\begin{align*}
  \|v\|_{Th} = \sup_{c} d\log_c(v)
\end{align*}

The Thurston metric is induced by this norm on the tangent bundle
\cite{Thu86}, and hence is a Finsler metric. 

We define the \textbf{unit norm sphere} at $X\in \mc{T}(S)$ by:
\begin{align*}
  \bm{S}_X = \{v \in T_X\mc{T}(S) : \|v\|_{Th} = 1\}
\end{align*}
Since the Thurston metric is induced by the 
Thurston norm, we think of $\bm{S}_X$ as the set of tangent vectors 
which arise from $1$-jets of $C^1$ geodesics starting at $X$. 
By claim 1.12 and Theorem 1.11 of \cite{PW22}, 
$\bm{S}_X$ consists of all unit tangent vectors at $X$.

More details on $d_v\log_{\alpha}$ and the characterization of 
the unit sphere can be found in subection \ref{subsec:stretch-vectors} 
and in \cite{Thu86, HOP21, DLRT20}.

Throughout this paper, we will work with two different co-ordinate 
systems for Teichm\"{u}ller space: Fenchel-Nielsen co-ordinates, 
and shearing co-ordinates. We review them in this 
section, following \cite{Mar16} for Fenchel-Nielsen co-ordinates, and 
following \cite{BBFS09} for shearing co-ordinates. For a more 
generalized overview of shearing co-ordinates in the case 
of a filling lamination that is not an ideal triangulation, 
we refer the reader to \cite{Ther14}.

\subsection{Shearing Co-ordinates \& Laminations}
Throughout this section, let $(X,\vphi)$ 
be a point in $\mc{T}(S)$. Let $\tilde X$ 
be the universal cover of $X$, which we will identify 
with $\HH^2$. A \textbf{geodesic lamination} $\lambda$ on 
$X$ is a closed subset of $X$ which can we decomposed 
as a disjoint union of (possibly bi-infinite) 
geodesics. If $(Y,\psi)$ is a different marking on $S$, 
then $\psi\vphi^{-1}(\lambda)$ gives a closed collection 
of disjoint arcs on $Y$, which we can turn into 
geodesics by an ambient isotopy, and hence, we can think of 
$\lambda$ as a lamination on 
$Y$. In this manner, $\lambda$ can be thought of as a 
lamination on the underlying surface $S$ without 
specifying a metric.

A geodesic lamination is called \textbf{complete} 
if its complementary components are triangles. A lamination 
$\lambda$ can be lifted to a lamination $\tilde \lambda$ 
on $\tilde X$, so $\lambda$ is complete if and only if 
$\tilde \lambda$ is a triangulation of $\HH^2$. A 
lamination $\lambda$ is called \textbf{chain-recurrent} 
if there exists a sequence of multicurves which 
converge to $\lambda$ in the Hausdorff topology. An 
alternate characterization of chain-recurrence is 
the following:

\begin{Definition}
  A geodesic lamination $\lambda$ is called \textbf{chain-recurrent} 
  if for any arc segment $I$ contained in $\lambda$, and for 
  any $\eps > 0$, there exists a geodesic simple closed curve $\alpha$
  in $S$ and an arc segment $J \subset \alpha$ 
  such that $d_H(I,J) < \eps$, where $d_H$ is the Hausdorff 
  distance.
\end{Definition}

In this sense, the space of chain-recurrent laminations 
is the closure of the space of multicurves, 
equipped with Hausdorff convergence \cite{DLRT20}. 

A \textbf{measured lamination} is a lamination 
$\lambda$ together with a measure, $\mu_{\lambda}$, defined on all 
arc segments intersecting $\lambda$ transversely 
whose endpoints lie on $\lambda$. We also require 
$\mu$ to be invariant under homotopy that moves 
the endpoints of arc segments along leaves of $\lambda$. 

As an example, consider a \textbf{weighted multicurve} -- 
a disjoint union of simple closed curves $\gamma_n$ with 
nonzero weights $a_n$. For any arc segment $I$, 
$\mu(I) = \sum_n a_n i(\gamma_n,I)$.

\begin{Definition}
  The space of projective measured laminations, $\PML(S)$, 
  is the space of all measured laminations up to 
  scaling the measure by a positive real number. 
\end{Definition}

We topologize $\PML(S)$ in the following manner: We 
say that $\lambda_n$ converge to $\lambda$ 
if for every arc segment $I$ in $S$, 
$\mu_{\lambda_n}(I)$ converges to $\mu_{\lambda}(I)$.

The following is useful for working with 
$\PML(S)$, and appears in Chapter 1 of \cite{HP92}:
\begin{Proposition}
  $\PML(S)$ is compact, and moreover, is the compactification of 
  the space of weighted multicurves with total weight $1$.
\end{Proposition}

We next explicitly describe shearing co-ordinates for 
a class of laminations that will be of interest 
later in this thesis.

Let $\lambda = \cup_{i=1}^{9g-9} \lambda_i$ be a complete geodesic
lamination consisting of $3g-3$ leaves which are simple closed curves
and $6g-6$ bi-infinite leaves.  Denote by $\mc{C}$ the closed leaves
of $\lambda$.  The leaves of $\lambda$ give an ideal triangulation of
$S$, and $\mc{C}$ form a pair-of-pants decomposition of $S$.

We wish to define a family of functions 
$S_{\lambda_i}: \mc{T}(S) \to \R$, which will give a 
co-ordinate system on $\mc{T}(S)$. To do this, 
we first define \textbf{shearing between 
triangles}.

Let $\Delta_1$ and $\Delta_2$ be two triangles in $S\ssm \lambda$, 
and let $\tilde\Delta_1$ and $\tilde\Delta_2$ be lifts of $\Delta_1$ 
and $\Delta_2$ to the universal cover $\tilde X = \HH^2$. Choose 
some geodesic $\gamma$ separating the interiors of $\tilde \Delta_1$ 
and $\tilde \Delta_2$, such that $\gamma$ is the geodesic 
between a vertex $v_1$ of $\tilde \Delta_1$ and to a vertex 
$v_2$ of $\tilde \Delta_2$ (if $\Delta_1$ and 
$\Delta_2$ are adjacent triangles on $S$, choose adjacent lifts, and 
let $\gamma$ be their intersection). 
We orient $\gamma$ so that the interior of $\Delta_1$ lies to 
the left of $\gamma$. 

An ideal triangle in $\HH^2$ has a well-defined 
incenter, incircle, and three distinguished medians 
along its edges, given by the intersection of 
the incircle with the edges. For each triangle 
$\tilde \Delta_i$, we define $m^i$ as the median of $\tilde \Delta_i$ 
lying on the edge separating $\gamma$ from 
$int(\tilde \Delta_i)$. 

Let $\vphi_i$ be the orientation-preserving 
parabolic isometry of $\HH^2$ fixing $v_i$ 
and sending $\gamma$ to the edge of $\tilde \Delta_i$ containing 
$m^i$. We set $q_i = \vphi^{-1}(m^i)$, 
and define $s(\Delta_1,\Delta_2)$ 
to be the signed distance between $q_1$ and 
$q_2$, where the sign is inherited from the 
orientation of $\gamma$.

In \cite{BBFS09}, it is proven that $s(\Delta_1,\Delta_2)$ 
is symmetric and does not depend on the geodesic $\gamma$. However, 
$s(\Delta_1,\Delta_2)$ as defined above 
may depend on the choice of lifts of $\Delta_1$ 
and $\Delta_2$, unless $\Delta_1$ and 
$\Delta_2$ are adjacent. 
\begin{figure}
  \begin{center}
    \begin{tikzpicture} [scale = 3]
      \draw[thick,<->] (-1.5,0) -- (1.5,0);
      \draw[thick,->] (0,0) -- (0,2.5);
      \draw[thick] (0.5,0) -- (0.5,2.5);
      \draw[thick] (-1,0) -- (-1,2.5);
      \draw[thick] (0,0) arc (0:180:0.5);
      \draw[thick] (0.5,0) arc (0:180:0.25);
      \draw[thick, dashed] (0.25,0.5) circle (0.25);
      \draw[thick, dashed] (-0.5,1) circle (0.5);

      \node at (-0.5, 2) {$\tilde\Delta_1$};
      \node at (0.25, 2) {$\tilde\Delta_2$};
      \node at (0.1,2.5) {$\tilde \gamma$};

      \draw[thick,|-|] (0,1) -- (0,0.5);
      \draw[thick,->] (-0.5, 1) -- (0,0.75);
      \node[fill=white] at (-0.5,1) {$s(\Delta_1,\Delta_2)$};
    \end{tikzpicture}
  \end{center}
  \caption{The set-up for computing shearing co-ordinates 
  between two adjacent triangles.}
\end{figure}
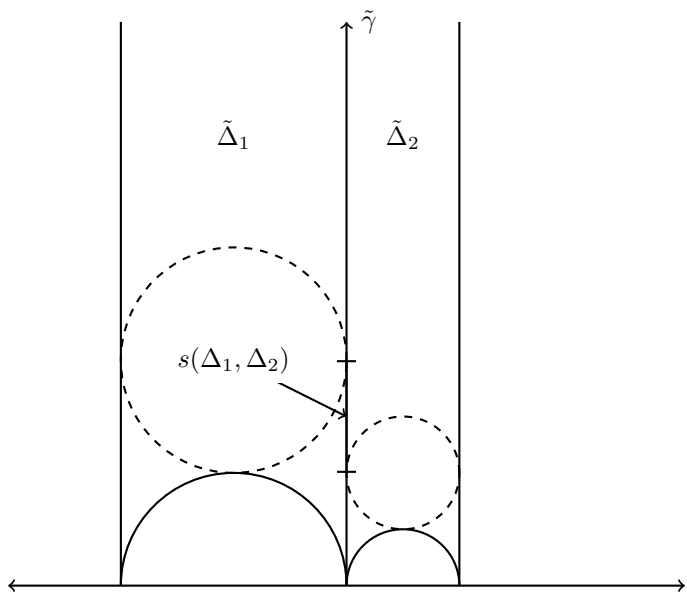

We are now ready to define $s_{\lambda_i}(X) = s(\Delta_1,\Delta_2)$, 
where $\Delta_i$ are chosen as follows:

\begin{itemize}
  \item If $\lambda_i$ is an infinite leaf of $\lambda$, 
    take $\Delta_1$ and $\Delta_2$ are the 
    (possibly coinciding) triangles adjacent to $\lambda_i$
  \item If $\lambda_i\in\mc{C}$, 
    we pick $\Delta_1$ and $\Delta_2$ to be 
    two triangles which asymptote to $\lambda_i$ 
    from both sides of $\lambda_i$.
\end{itemize}

\begin{Lemma}
  Let $c\in\mc{C}$, and let 
  $\lambda_c\subset \lambda$ be the sublamination of 
  $\lambda$ consisting of leaves converging to $c$. 

  Then $s_{c}(X)$ is well-defined 
  up to $\Z[\{s_{\lambda_c}\}]$.
\end{Lemma}

\begin{proof}
  This follows from Lemma 3.1 of \cite{BBFS09}.
\end{proof}

Choosing prescribed 
lifts for every triangle in $S\ssm\lambda$, 
we get a well-defined \textbf{shearing co-ordinates} map
$\mc{T}(S_g) \to \R^{9g-9}$ given by: 
\begin{align*}
  X \to \{s_l\}_{l\subset \lambda \textrm{ is a leaf}}
\end{align*}

The shearing co-ordinate map does not give an isomorphism, 
as for any $\lambda$ and $X$, the collection 
$\{s_{\lambda_i}(X)\}$ is linearly dependent (see 
Lemma 3.2 in \cite{BBFS09}). Choosing a 
linearly independent subset of these gives a 
homeomorphism from $\mc{T}(S_g)$ to $\R^{6g-6}$ 
\cite{BBFS09}.

\subsection{Twisting Co-ordinates}

Let $\mc{C}$ be a pair-of-pants decomposition 
of a surface $S$. In the same way as in our 
definition for the shearing co-ordinates, 
it will be convenient to choose prescribed 
lifts of every $c\in \mc{C}$ in $\tilde X$. For any 
closed leaf $c\in \mc{C}$, let $P_1(c)$ and $P_2(c)$ 
be the (possibly same) pairs of pants 
adjacent to $c$. Let 
$\tilde c$ be a lift of $c$, and choose 
lifts, $\tilde P_i(c)$ of $P_i(c)$ which are adjacent 
to $\tilde c$ on either side of $\tilde c$.

Let $c\neq c_1$ and $c \neq c_2$ be cuffs of $P_1(c)$ and $P_2(c)$ 
respectively, and let $\tilde c_i$ be the lifts of $c_i$ 
bounding $\tilde P_i(c)$.

Let $\eta_i$ be the unique simple geodesic 
segment intersecting $\tilde c_i$ with start point on $\tilde c$ and 
end point on $\tilde c_i$, which is normal 
to both $\tilde c$ and $\tilde c_i$. Let $p_i(c)$ 
denote the start-point of $\eta_i$. 

We define the \textbf{twist co-ordinate relative to $c$}, 
$\tau_c(X)$ to be the signed distance between $p_1(c)$ and $p_2(c)$, 
so that the sign is \textit{positive} if we turn \textit{left} 
to get from $p_1$ to $p_2$. Choosing different 
lifts of $c$ will result in twist co-ordinates 
that differ by an integer multiple of $l_c(X)$.

Additionally, it can be shown that choosing 
different curves $c_i$ in the pair-of pants 
will change $p_i$ by 
a Gaussian-integer multiple of $l_c$, and 
hence, $\tau_c(X)$ is well defined up to 
an integer multiple of $l_c(X)$.

The map $FN:\mc{T}(S) \to \R^{6g-6}$ 
given by $X \to \{l_{c}(X),\tau_{c}(X)\}_{c\in\mc{C}}$ 
is a homeomorphism \cite{Mar16}, so we call 
this map the \textbf{Fenchel-Nielsen co-ordinate} system 
on $\mc{T}(S)$.

\subsection{Stretch Laminations, Chain-Recurrence, and Thurston 
Geodesics}

\begin{Definition}
  Let $I$ be a possibly infinite closed interval. 
  A \textbf{forward} (resp. \textbf{backwards}) 
  geodesic is a map $\gamma:I\to M$ 
  satisfying $d_{Th}(\gamma(s),\gamma(t)) = t-s$ 
  (resp. $d_{Th}(\gamma(t),\gamma(s)) = t-s$) 
  for all $t\ge s$.

  If $I=[a,b]$ is finite, we say that $\gamma$ 
  starts at $\gamma(a)$ and ends at $\gamma(b)$.
\end{Definition}

Throughout this thesis, a ``geodesic'' means a 
forward geodesic, unless otherwise stated.

Given $X,Y\in \mc{T}(S)$, we define 
\begin{align*}
  \Outenv(X) &= 
  \{Z\in \mc{T}(S) : Z \textrm{ lies on a forward geodesic 
  starting at } X\}\\ 
  \In(Y) &= 
  \{Z\in \mc{T}(S) : Z \textrm{ lies on a backwards geodesic 
  starting at } Y\}
\end{align*}

For $X,Y\in \mc{T}(S)$, we define the 
\textbf{geodesic envelope}, $\Env(X,Y) = \Outenv(X) 
\cap \In(Y)$. 

Given $X,Y\in\mc{T}(S)$, we can consider a sequence of multicurves
$\alpha_i$ such that $d_{Th}(X,Y) = \lim_{i\to\infty}
\frac{l_{\alpha_i}(Y)}{l_{\alpha_i}(X)}$. The space of geodesic
laminations equipped with the Hausdorff topology is compact
\cite{BZ05}, so up to subsequence, $\alpha_i$ converges to some
geodesic lamination $\lambda$. It turns out \cite{Thu86} that the
union of all Hausdorff limits of subsequences of $\{\alpha_i\}_n$ is
itself a geodesic lamination, call it $\lambda_{\{\alpha_i\}}$.
Moreover, if $\alpha'_i$ is another sequence whose length ratio
converges to $d_{Th}(X,Y)$ then any Hausdorff limit of $\{\alpha'_i\}$
is disjoint from $\lambda_{\{\alpha_i\}}$. 

Thus, it makes sense to define the \textbf{maximally stretched
lamination}, $\Lambda(X,Y)$ as the union of all Hausdorff limits of
sequences of multicurves whose length ratio converges to
$d_{Th}(X,Y)$.

Given $X_0\in \mc{T}(S)$, and any completion of a 
maximal chain-recurrent lamination $\lambda$, 
there exists an analytic 1-parameter 
family of metrics $X_t = \textrm{Stretch}(X_0,\lambda,t) 
\subset \mc{T}(S)$ with the following properties:
\begin{itemize}
  \item $l_{\lambda}(X_t) = e^tl_{\lambda}(X_0)$
  \item For $0\le s\le t$, $d_{Th}(X_s,X_t) = t-s$
  \item $\Lambda(X_s,X_t) = \lambda$
\end{itemize}

This family of metrics on $S$ is called the \textbf{Thurston Stretch
Path} associated to $X_0$ and $\lambda$. In particular, when
$\Lambda(X,Y)$ is maximal amongst all chain-recurrent laminations,
there is a unique geodesic from $X$ to $Y$, and the points along it
are precisely the points in $\Env(X,Y)$ \cite{DLRT20}.

\begin{Remark}
  Let $X_t$ be some smooth $1$-parameter family of 
  surfaces in $\mc{T}(S)$.

  The maps $s_{\lambda}:X_t \to \R^{9g-9}$ and 
  $FN:X_t \to \R^{6g-6}$ are only well-defined 
  up to the choices of lifts in their construction. However, 
  the maps $\dot s_{\lambda}:\mc{T}(S) \to \R^{9g-9}$ 
  and $\dot{FN}:X_t \to \R^{6g-6}$ defined by 
  postcomposition of $s_{\lambda}$ and $FN$ 
  by differentiation with respect to $t$ are well 
  defined.
\end{Remark}

Let $\lambda$ be some chain-recurrent lamination, 
and let $X,Y\in \mc{T}(S)$. We 
define
\begin{align*}
  \Outenv(X,\lambda) &= 
  \{Z\in \Outenv(X) : \lambda \subset \Lambda(X,Z)\}\\
  \In(Y, \lambda) &= 
  \{Z\in \In(Y) : \lambda \subset \Lambda(Y,Z)\}
\end{align*}

By this definition, $\Env(X,Y) = \Outenv(X,\Lambda(X,Y))
\cap \In(Y,\Lambda(X,Y))$. The geodesic envelope can be
thought of as a $1$-parameter family of cross-sections, where at any
time $t>0$, $\Env_t(X,Y)$ consists of all points of distance
$t$ from $X$ lying along geodesics from $X$ to $Y$.

\begin{Definition}
  Let $X,Y\in \mc{T}(S)$. We define the \textbf{width} 
  of the envelope $\Env(X,Y)$ as:
  \begin{align*}
    w(X,Y) &= \sup_{g_1,g_2 \mc{G}(X,Y)} \sup_{t\in [0,d_{Th}(X,Y)]} 
    d_{Th}(g_1(t), g_2(t))\\
    &= \sup_t \textrm{Diam}(\textrm{Env}_t(X,Y))
  \end{align*}
\end{Definition}


\section{The Infinitesimal Envelope}
\label{sec:infinitesimal-envelope}

The goal of this section is to understand the infinitesimal 
structure of geodesic envelopes in the Thurston metric. 

We begin the section by proving Theorem \ref{thm:stretch-vectors-extremal},
which will be one of the main ingredients in the arguments employed in
Section \ref{sec:bounded-width} to prove Theorem \ref{thm:bounded-envelope}

We continue by examining $\Env_0(X,Y)$ 
when $\Lambda(X,Y)$ contains a pair-of-pants decomposition of $S$. In
this case, the envelope width is determined by twist parameters along
geodesics. This will then allow us to estimate the maximal and minimal
twisting within $\Env(X,Y)$, which we will use in
Section \ref{sec:bounded-width}. 

\subsection{Stretch Vectors in the Envelope}
\label{subsec:stretch-vectors}

In this section, we study the set of tangent vectors in $T_X\mc{T}(S)$
which arise from $1$-jets of $C^1$ geodesics starting at some point
$X\in \mc{T}(S)$.  This set, called the \textbf{unit norm sphere}, and
denoted by $\bm{S}_X$ was shown in \cite{HOP21} to have
a combinatorial structure of a convex body with a convex
stratification whose faces come from chain-recurrent laminations 
in $S$. If $v\in \bm{S}_X$ is 
a $1$-jet of a stretch path $\mathrm{Stretch}(X,\lambda,t)$, we call
it a \textbf{stretch vector with respect to $\lambda$}, and 
denote it by $v = v_{\lambda}(X)$.

The following lemma appears in various parts of \cite{Thu86} 
and \cite{HOP21}. We adapt it to our language:
\begin{Lemma}
  For any $X\in \mc{T}(S)$, $\bm{S}_X$ is a topological sphere around
  $0 \in T_X\mc{T}(S)$. 
\end{Lemma}
Moreover, $\bm{S}_X$ is an infinite union 
of convex sets, called \textbf{faces}, which glue together 
in a combinatorial way. 
In \cite{HOP21}, it is shown that there is a one-to-one 
correspondence with topological chain-recurrent laminations 
and faces of $\bm{S}_X$. We review their definitions 
and theorems later in this section.

We further analyze the structure of $\bm{S}_X$, and
prove Theorem \ref{thm:stretch-vectors-extremal}, 
answering Conjecture 1.12 of \cite{HOP21} in the affirmative:

\begin{Theorem*}
  The set of stretch vectors in $\bm{S}_X$ with respect to completions
  of maximal chain-recurrent laminations is precisely the set of
  extreme points in $\bm{S}_X$
\end{Theorem*}

This result not only allows us to characterize 
faces in $\bm{S}_X$ using stretch paths, but 
also allows us to find extreme points in 
the \textbf{infinitesimal envelope} from $X$ 
to $Y$. These extreme points can later be used to 
give upper and lower bounds on the twisting width 
between \textit{any} two geodesics from $X$ to $Y$, not 
just stretch lines starting at $X$, as 
computed in Section \ref{sec:bounded-width}.

Throughout this section, if $\Lambda$ is some lamination, we fix the
following notation:
\begin{itemize}
  \item We denote by $\Lambda^{CR}$ the largest sublamination of
    $\Lambda$ which is chain-recurrent (see Section \ref{sec:intro}).
  \item If $\Lambda$ is chain-recurrent, we denote by 
    $CR(\Lambda)$ to be the set of chain-recurrent laminations 
    containing $\Lambda$ that are maximal with respect to inclusion. 
    Abusing notation, we write $CR(\emptyset)$ to denote 
    the set of all chain-recurrent laminations that are maximal 
    with respect to inclusion.
  \item If $\Lambda$ is chain-recurrent or is empty, we define:
    \begin{align*}
      MCR(\Lambda)
      =\{\mu\in CR(\Lambda) : \forall \nu \in CR(\Lambda), 
      \mu \subset \nu \Rightarrow \mu = \nu\}
    \end{align*}
  \item If $\Lambda$ is chain-recurrent or empty, we denote by
    $M(\Lambda)$ to be the set of completions of laminations in
    $MCR(\Lambda)$. 
  \item If $\Lambda$ is chain-recurrent or empty, and 
    $X\in \mc{T}(S)$, we define
    $SV_X(\Lambda) = \{v_{\mu}(X):\mu \in M(\Lambda)\}$.  We abuse
    notation and write $SV_X = SV_X(\emptyset) = \{v_{\mu}(X): \mu\in
    M(\emptyset)\}$
\end{itemize}

Using this notation, Theorem \ref{thm:stretch-vectors-extremal} 
says that the set of extreme points of $\bm{S}_X$ 
is precisely $SV_X(\emptyset)$. We prove some lemmas and state some 
facts about the above sets and spaces.

\begin{Fact} \label{fact:CR-closed}
  If $\Lambda$ is chain-recurrent or empty, $CR(\Lambda)$ 
  is compact in the Hausdorff topology.
\end{Fact}

\begin{proof}
  The space of all geodesic laminations on a surface is 
  compact \cite{BZ05}, so it suffices to show that 
  the set of chain-recurrent laminations containing 
  $\Lambda$ is closed. Note that containment of $\Lambda$ is a closed
  condition, since laminations are closed subsets of $S$. Let 
  $\Lambda_n$ be a converging sequence of chain-recurrent laminations.
  If $\alpha_n^k$ is a sequence of simple closed multicurves
  converging in the Hausdorff topology to $\Lambda_n$, then
  $\alpha_n^n$ is a sequence of simple closed curves converging to
  $\Lambda$.
\end{proof}

The following lemma follows from the proof of Theorem 8.5 in 
\cite{Thu86}, and also appears as Corollary 2.3 in \cite{DLRT20} 
and in the discussion following Remark 2.9 in \cite{HOP21}.

\begin{Lemma} \label{lem:v_lambda-well-defined}
  If $\Lambda$ is chain-recurrent or empty, 
  and if $\lambda\in MCR(\Lambda)$, then 
  for any two completions $\lambda_1$ and 
  $\lambda_2$ of $\lambda$, we have that 
  for any $X\in \mc{T}(S)$, and $t \ge 0$, 
  $\mathrm{Stretch}(X,\lambda_1,t)=\mathrm{Stretch}(X,\lambda_2,t)$, 
  and in particular, 
  $v_{\lambda_1}(X) = v_{\lambda_2}(X)$
\end{Lemma}

\subsection{The case for $\Env_0(X,Y)$}
Let $X,Y \in \mc{T}(S)$ be two points in Teichmuller space, and let
$\mc{G}(X,Y)$ denote the set of all  geodesics parametrized by arc
length from $X$ to $Y$. By definition, $\cup \mc{G}(X,Y) = \Env(X,Y)$.
If $g\in \mc{G}(X,Y)$, we define the \textbf{$1$-jet} of $g$ by
$v_g(X) = \frac{d}{dt}|_{t=0} g(t)$. Note that $v_g(X)\in T_X
\mc{T}(S)$ is a unit tangent vector.

\begin{Definition}
  The \textbf{infinitesimal envelope}, $\Env_0(X,Y)\subset
  T_X\mc{T}(S)$ is defined by:
  \begin{align*}
    \Env_0(X,Y) = \{v_g(X) : g\in \mc{G}(X,Y)\}
  \end{align*}
\end{Definition}

For any $X,Y\in \mc{T}(S)$, we denote by $\Lambda(X,Y)$ the maximally
stretched lamination between $X$ and $Y$. In this subsection, we prove
the following:
\begin{Theorem} ~\label{thm:conv-hull-env0}
  For any $X,Y\in \mc{T}(S)$, $\Env_0(X,Y)$ is the convex hull of:
  $SV_X(\Lambda(X,Y))$. Moreover, 
  $SV_X(\Lambda(X,Y))$ is precisely the set of extremal vectors 
  in $\Env_0(X,Y)$.
\end{Theorem}

\begin{Example}
  As an example for Theorem \ref{thm:conv-hull-env0}, 
  consider the genus $2$ surface, $S_2$, and let 
  $\lambda = \alpha \cup \beta \cup \gamma$ be a pair-of-pants
  decomposition consiting of non-separating curves. 
  Let $X \in \mc{T}(S_2)$ be arbitrary. 
  Working in Fenchel-Nielsen co-ordinates with respect 
  to $\alpha,\beta,\gamma$, $\Outenv(X,\lambda)$ 
  can be thought of as a $3$-dimensional cone lying 
  in $\R^{6}$. In particular, every $Z$ in $\Outenv(X,\lambda)$ 
  of distance $t$ from $X$ has the same length co-ordinates, 
  and the only interesting co-ordinates are the twists. 

  If $Y\in \Outenv(X,\lambda)$, we 
  can consider the $3$-dimensional projection of 
  $\Env_0(X,Y)$ to the tangent subspace of $T_X\mc{T}(S_2)$ 
  spanned by the directions corresponding to 
  twist co-ordinates. Theorem \ref{thm:conv-hull-env0} 
  then says that the extremal vectors in 
  this projection must be stretch vectors with 
  respect to laminations in $M(\lambda)$.

  Using the formulas developed in subsection \ref{subsec:twist-compute}, 
  we can plot all $1$-jets of stretch paths 
  emenating from $X$ and maximally-stretching $\lambda$. This 
  is Figure \ref{fig:chamf-cube}.

  \begin{figure}
    \begin{center}
      \includegraphics[width=0.6\textwidth]{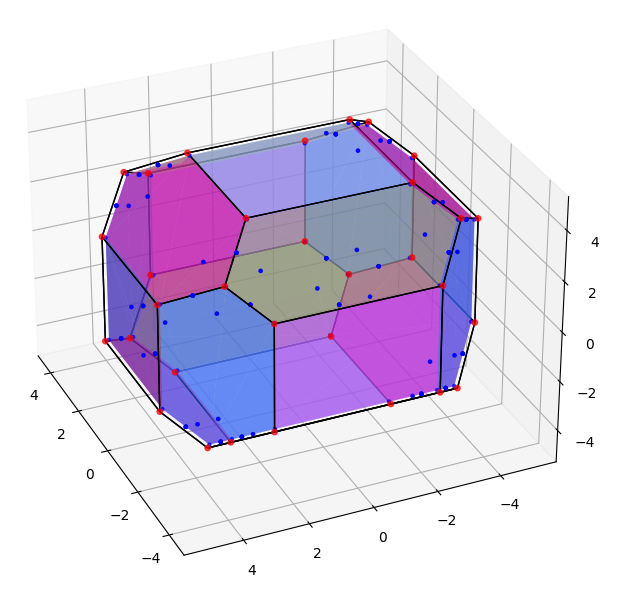}
    \end{center}
    \caption{The projection of all stretch vectors emenating from 
    $X$ and maximally stretching $\lambda$. Combinatorially, 
    this is a chamfered cube.}
    \label{fig:chamf-cube}
  \end{figure}
  The red dots in the above picture are stretch vectors 
  corresponding to the $32$ chain-recurrent completions 
  of $\lambda$. The rest of the points are 
  stretch paths corresponding to non chain-recurrent 
  laminations. The eight red vertices at the ``corners'' of 
  the projected infinitesimal envelope correspond 
  to the completions $\Lambda$ of $\lambda$ that have the property
  that for any pair of curves in $\lambda$, there exists a leaf 
  $l \subset \Lambda$ asymptotic to both. There are 
  precisely $8$ such laminations, 
  corresponding to the $2^3$ possible directions 
  in which leaves can asymptotically twist around 
  $\alpha,\beta$, and $\gamma$.
  that have leaves 
\end{Example}

We will split the proof of Theorem \ref{thm:conv-hull-env0} up into three 
main lemmas:

\begin{Lemma} \label{lem:conv-hull-finite-completions}
  For any $X,Y\in \mc{T}(S)$, if 
  $\Lambda(X,Y)$ has finitely-many completions, 
  then $Env_0(X,Y)$ is the convex hull of 
  $SV_X(\Lambda(X,Y))$
\end{Lemma}

We will use this lemma to show the general case:
\begin{Lemma} \label{lem:conv-hull-all-completions}
  For any $X,Y\in \mc{T}(S)$, $Env_0(X,Y)$ is the convex hull of:
  $SV_X(\Lambda(X,Y))$
\end{Lemma}

Finally, we show:

\begin{Lemma} \label{lem:stretch-vectors-not-internal}
  Let $\lambda \in M(\Lambda(X,Y))$, 
  and let $v_{\lambda} \in SV_X(\Lambda(X,Y))$ be 
  a convex combination:
  $v_{\lambda} = \sum_i a_i v_{\lambda_i}$, where 
  $\lambda_i \in M(\Lambda(X,Y))$. Then 
  $v_{\lambda_i} = v$ for all $i$.
\end{Lemma}

Before we prove the lemmas, we will prove a helpful technical lemma
about oriented foliations and their $1$-jets. When we say 
\textit{oriented foliation}, we mean a foliation 
on some manufold $M$ such that every leaf carries with it an
orientation, and that these orientations vary continuously along any
path transverse to the foliation.

Let $\mc{F}_1,\mc{F}_2,\hdots$ be smooth oriented foliations defined
on some open domain with smooth (possibly empty) boundary $0\in
U \subset \R^k$. For each oriented foliation, we can construct a 
vector field of unit-length vectors tangent to the foliation. Such
a vecor field can be constructed by taking the associated line field
of the foliation and assigning a direction at every point using the
orientation of the foliation. We will refer to the flow along the
vector field associated to the foliation by the `flow along the
foliation'.

Let $U^{\eps} = \{y\in U: d(y,\del y) >\eps\}$ 
and let $f_i:[0,\eps)\times U^{\eps} \to \R^k$ 
be defined by setting $f_i(t,x)$ to be the flow for time $t$ along
$\mc{F}_i$ starting from $x$. Taking $\eps$ sufficiently small, 
we can always ensure that $0 \in U^{\eps}$, and that 
all of the flows $f_i$ are defined at $x=0$.

We prove:
\begin{Lemma} ~\label{lem:flows-lemma}
  Let $\alpha:[0,T) \to U$ be a path differentiable 
  at $0$, and such that 
  $\alpha(0) = 0$. Assume that for every 
  $t< T$, there exist 
  $t_1,\hdots,t_n\ge 0$ and $i_1,\hdots,i_n$ such that 
  $t = \sum_i t_i$ and $\alpha(t)
  = f_{i_1}(t_1,f_{i_2}(t_2,(\hdots,(f_{i_n}(t_n,0)))))$. 
  Furthermore, assume that the non-negative span,
  $\Span_{\ge 0}(\{f_i\}_i)$ is
  a closed subset of $\R^N$. Then $\alpha'(0) \in \Span_{\ge
  0}(\{\mc{F}_i'(0)\}_{i}) = \Span_{\ge 0}(\{\frac{\del}{\del t}f_i
  (0,0)\}_{i})$ 
\end{Lemma}

\begin{figure}
  \begin{center}
    \includegraphics[width=0.6\textwidth]{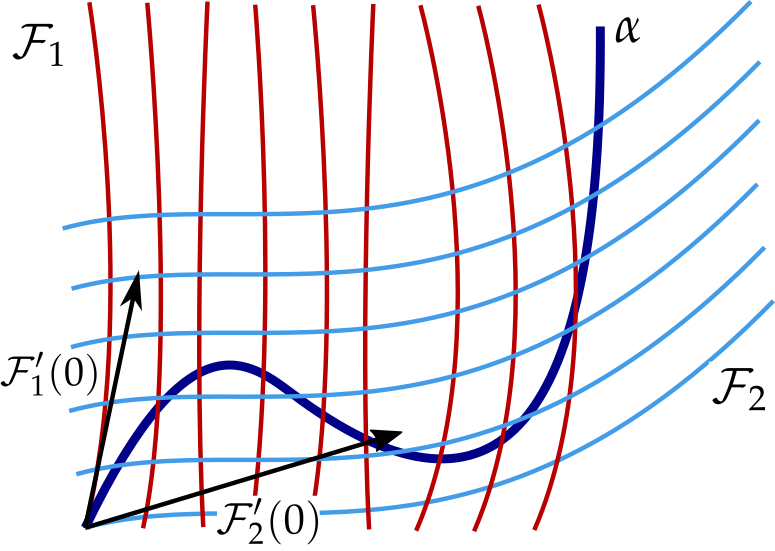}
  \end{center}
  \caption{The picture of the set-up in \ref{lem:flows-lemma}}
\end{figure}

\begin{proof}
  In order to estimate $\alpha'(0)$, we first note that 
  as $t$ gets really small, so must the $t_i$'s. Thus, 
  for small $t$, we can take a first-order expansion of
  $\alpha(t) = f_{i_1}(t_1,f_{i_2}(t_2,(\hdots,(f_{i_n}(t_n,0)))))$ when $t_1$ is small to get:
  \begin{align*}
    \alpha(t) &= f_{i_1}(0,f_{i_2}(t_2,(\hdots,(f_{i_n}(t_n,0))))) \\
    &\qquad   + t_1 \frac{\del f_{i_1}}{\del t}(0, f_{i_2}(t_2,(\hdots,(f_{i_n}(t_n,0)))))
    + o(t_1^2)\\ 
    &= f_{i_2}(t_2,(\hdots,(f_{i_n}(t_n,0))))
    + t_1 \frac{\del f_{i_1}}{\del t}(0, f_{i_2}(t_2,(\hdots,(f_{i_n}(t_n,0)))))
    + o(t_1^2)
  \end{align*}
  Where we understand $\frac{\del f_{i_1}}{\del t}$ to be the derivative 
  of $f_{i_1}$ with respect to the first co-ordinate.

  Expanding the $f_{i_2}(t_2,(\hdots,(f_{i_n}(t_n,0))))$ in terms 
  of $t_2$, we get:
  \begin{align} \label{eq:alpha_est1}
    \alpha(t) &= 
    f_{i_3}(t_3,(\hdots,(f_{i_n}(t_n,0))))
    + t_2 \frac{\del f_{i_2}}{\del t}(0, f_{i_3}(t_3,(\hdots,(f_{i_n}(t_n,0))))) \\
    &\quad + t_1 \frac{\del f_{i_1}}{\del t}(0,
    f_{i_2}(t_2,(\hdots,(f_{i_n}(t_n,0))))) + o(t^2)
  \end{align}

  If we continue to expand in this manner, we 
  obtain:
  \begin{align} \label{eq:alpha_est2}
    \alpha(t) &= f_{i_n}(t_n,0)
    +\sum_{j=1}^{n-1} t_j \frac{\del f_{i_j}}{\del t}
    (0,f_{i_{j+1}}(t_{j+1},(\hdots,(f_{i_n}(t_n,0))))) + o(t^2)\\
    &= t_n \frac{\del f_{i_n}}{\del t}(0,0)
    +\sum_{j=1}^{n-1} t_j \frac{\del f_{i_j}}{\del t} (0,f_{i_{j+1}}(t_{j+1},(\hdots,(f_{i_n}(t_n,0))))) + o(t^2)\\
  \end{align}

  We denote $v_j = \frac{\del}{\del t} f_j (0,0)$, 
  and notice that $f_j(t_j,0) = t_j v_j + o(t^2)$. 
  Thus, 
  \begin{align*}
    f_{i_{n-1}}(t_{n-1},f_{i_n}(t_n,0)) &= f_{i_{n-1}}(t_{n-1}, t_n
    v_{i_n} + o(t^2))\\
    &= f_{i_{n-1}}(t_{n-1},0) + D_0 f_{i_{n-1}}(t_{n-1},0) \cdot (t_n
    v_{i_n})
    + o(t^2))\\
    &= t_{i_{n-1}}v_{n-1} + t_n D_0 f_{i_{n-1}}(t_{n-1},0)\cdot
    v_{i_n} + o(t^2)\\
    &= t_{i_{n-1}}v_{n-1} + t_n D_0 (t_{n-1} v_{i_{n-1}} + o(t^2))
    \cdot v_{i_n}
    + o(t^2)\\
    &= t_{n-1}v_{i_{n-1}} + t_n t_{n-1} v_{i_{n-1}} + o(t^2) \\
    &= t_{n-1}v_{i_{n-1}} + o(t^2)
  \end{align*}
  Where we used the fact that all functions are smooth and 
  hence have bounded derivatives in a neighbourhood 
  of $0$. Note that we can continue this computation 
  to replace the composition terms in
  ~\ref{eq:alpha_est2} with simpler, linear terms:

  \begin{align} \label{eq:alpha_est3}
    \alpha'(0) &= t_n v_{i_n} +\sum_{j=1}^{n-1}
    t_j \frac{\del f_{i_j}}{\del t}
    (0,t_{j+1} v_{i_{j+1}} + o(t^2)) + o(t^2)\\
    &= t_n v_n + \sum_{j=1}^{n-1} t_j \left(v_{i_j} + D_0 f_{i_j}(0,0) \cdot
    (t_{j+1} v_{i_{j+1}} + o(t^2))\right) + o(t)\\
    &= \sum_{j=1}^n t_j v_{i_j} + o(t^2)
  \end{align}

  Since $W = \Span_{\ge 0}(\{v_i\}_i)$ is closed, 
  we also get that for any $t > 0$, we have 
  that $d(\alpha(t), W) = o(t^2)$. This means 
  that $\alpha'(0)$ is in $W$.
\end{proof}

\begin{Lemma} \label{lem:finite-stretch-vectors-span-hull}
  Assume that $\Lambda(X,Y)$ 
  has finitely-many completions, 
  and let $v \in \Env_0(X,Y)$. Then $v$ is a convex 
  combination of vectors in $SV_X(\Lambda(X,Y))$.
\end{Lemma}

\begin{proof}

  Let $v(t):[0,1] \to \mc{T}(S)$ be a geodesic path parametrized by
  arc length whose $1$-jet at $0$ is $v$. Since $v\in \Env_0(X,Y)$, we
  can freely assume that $v(t)$ lies in $\Env(X,Y)$ for any
  sufficiently small $t$.

  Working in co-ordinates $\vphi:\mc{T}(S_g) \to \R^{6g-6}$, for each 
  completion $\lambda$ of $\Lambda(X,Y)$, we get a foliation of 
  $\R^{6g-6}$ defined by the Thurston stretch line 
  corresponding to $\lambda$. This foliation also 
  comes with a natural \textit{flow direction} by 
  taking the forward corresponding to the stretch 
  path defined by $\lambda$.

  Note that when $\lambda$ contains a pair-of-pants decomposition,
  shearing co-ordinates are smoothly related to Fenchel-Nielsen
  co-ordinates on $\mc{T}(S_g)$ by the computations done in
  subsection \ref{subsec:twist-compute}. More generally, by Theorem A of
  \cite{Bon01} and computations in \cite{Gen15}, 
  we have that length functions of simple closed 
  curves are smooth in the shearing co-ordinates. Since 
  Teichmuller space is locally parametrized by length 
  functions of $6g-6$ simple closed curves, it follows 
  that we can think of shearing co-ordinates not just 
  as topological co-ordinates on $\mc{T}(S_g)$, 
  but as smooth co-ordinates as well. Stretch lines are smooth in
  shearing co-ordinates, since they can be realized as rays starting at the origin in the shearing co-ordinates corresponding to the lamination defining the stretch path \cite{Thu86}. Thus, stretch lines are also smooth in Fenchel-Nielsen co-ordinates. In
  particular, it follows that if $\lambda$ is a completion of a 
  maximal chain-recurrent lamination containing $\Lambda(X,Y)$, we get
  a smooth oriented foliation (which we will call $\mc{F}_{\lambda}$)
  of $\R^{6g-6}$ under any co-ordinates on Teichm\"{u}ller space.

  By Thurston's construction of geodesics using concatenation of
  stretch lines, it follows that for each $t$, $v(t)$ can be expressed
  as a concatenation of flows along the foliations $\mc{F}_{\lambda}$.
  Moreover, a careful reading of Theorem 8.5 of \cite{Thu86} actually
  says that we can choose these laminations to lie in
  $M(\Lambda(X,Y))$. 

  By Lemma \ref{lem:flows-lemma}, using the fact 
  that $SV_X(\Lambda(X,Y))$ is finite and hence $\Span_{\ge
  0}(SV_X(\Lambda(X,Y)))$ is closed, we get that $v\in \Span_{\ge
  0}(SV_X(\Lambda(X,Y)))$ and we write $v = \sum_{i=1}^{N} a_i
  v_{\lambda_i}$, where $\lambda_i \in M(\Lambda(X,Y))$.

  Next, we show that $v$ is a convex linear combination of 
  $SV_X(\Lambda(X,Y))$.
  We write $v = \sum_i a_i v_{\lambda_i}$, 
  and consider the path:
  \begin{align*}
    \beta(t) = f_1(k a_1t,f_2(ka_2t,(\hdots,(f_n(ka_nt,0)))))
  \end{align*}
  Where $k$ is chosen such that $k\sum a_i = 1$, and 
  where we denote $f_i(t,X) = \textrm{Stretch}(X,\lambda_i,t)$. 
  Note that $\beta$ is a length-parametrized geodesic, since
  $\Lambda(X,Y)$ is maximally stretched along it. By the 
  same computations in Lemma \ref{lem:flows-lemma}, we get 
  that $\beta'(0) = k v$. In particular, 
  since $v$ and $\beta'(0)$ are unit-length vectors, 
  it follows that $k = 1$, and $\sum_i a_i = 1$. 
\end{proof}

We now treat the cases when $\Lambda(X,Y)$ 
has infinitely many completions.

\begin{Lemma} \label{lem:cone-SV-closed}
  If $\Lambda$ is chain-recurrent or 
  empty, then $\Span_{\ge 0}(SV_X(\Lambda))$ 
  is closed.
\end{Lemma}

\begin{proof}
  We will show that for $\lambda_n \in MCR(\Lambda)$, if
  $v_{\lambda_n} \in SV_X(\Lambda)$ converge to some $v$, then $v$ is
  in the convex hull of $SV_X(\Lambda)$. By \ref{fact:CR-closed}, up
  to subsequence, there exists some $\lambda \in CR(\Lambda)$ such
  that $\lambda_n \to \lambda$ in the Hausdorff topology.

  By maximality of $\lambda_n$, 
  $S\ssm \lambda_n$ has at most finitely-many 
  completions to a triangulation, and 
  since $\lambda$ is a Hausdorff limit of $\lambda_n$, it shares this
  property. Let $\lambda^1,\hdots,\lambda^k \in M(\lambda)$ be the
  finitely-many completions of maximal chain-recurrent laminations 
  containing $\lambda$.

  Since the complementary components of $\lambda_n$ 
  stabilize close to the complementary components of 
  $\lambda$, it follows that for any sufficiently large $n$, 
  there exist completions of $\lambda_n^i \in M(\lambda_n)$ 
  which Hausdorff converge to $\lambda^i$. We can do 
  this, for example, by adding in leaves 
  into the complementary components of $\lambda_n^i$ 
  which converge to the added leaves of $\lambda^i$.

  For any $t$, we consider the family 
  $S_n(t) = \mathrm{Stretch}(X, \lambda_n^i,t)$. Since 
  $S_n(t)$ are smooth geodesics, it follows that 
  they have uniformly bounded first derivatives. By 
  the Arzela-Ascoli theorem, up to subsequence, $S_n(t)$ 
  converge to a continuous path, $S_{\infty}(t)$ in 
  $\mc{T}(S)$, starting at $X$. Moreover, $S_{\infty}$ 
  is differentiable at $0$, and has derivative 
  equal to $v$. Note that 
  because $\lambda_n^i$ converge to $\lambda^i$, 
  and $\lambda \subset \lambda^i$, 
  it follows that $\lambda \subset \Lambda(X,S_{\infty}(t))$, 
  and so $\Lambda(X,S_{\infty}(t))$ has 
  finitely-many completions.

  By \ref{lem:conv-hull-finite-completions}, 
  we have that $\Env_0(X, S_{\infty}(t))$ 
  is the convex hull of $SV_X(\Lambda(X,S_{\infty}(t)))$, 
  which is closed. 
  Since this is true for all $t$, a similar 
  Taylor series argument as in the previous 
  lemmas above shows that, 
  in fact, $v = \frac{d}{dt}|_{t=0} S_{\infty}(t)$ 
  lies in the convex hull of $SV_X(\Lambda(X,S_{\infty}(t)))$, 
  as desired.
\end{proof}

The proof of Lemma \ref{lem:conv-hull-all-completions} 
follows in the same way as the proof of 
Lemma \ref{lem:conv-hull-finite-completions}, 
where we now use Lemma \ref{lem:cone-SV-closed} 
to meet the conditions of 
Lemma \ref{lem:flows-lemma}.

\begin{Remark}
  The reason we split Lemma \ref{lem:conv-hull-all-completions} 
  into two lemmas is a bit of a subtlety. In the argument 
  above, we showed that the non-negative 
  span of $SV_X(\Lambda(X,Y))$ is closed. A priori, 
  the entire proof of Lemma \ref{lem:cone-SV-closed} 
  could have been skipped if we knew that $SV_X(\Lambda(X,Y))$ 
  was closed.

  To argue this, let $v_{\lambda_n}$ be a 
  converging sequence of stretch vectors 
  converging to some $v$. One should show that the Hausdorff 
  limit of $\lambda_n \in MCR(\Lambda(X,Y))$ is a 
  lamination $\lambda \in MCR(\Lambda(X,Y))$, and then use 
  the fact that stretch vectors change continuously 
  in their defining laminations to get that 
  $v = v_{\lambda}$. It is easy to see that 
  $\lambda$ must be chain-recurrent, and we must show that it 
  is maximal amongst all chain-recurrent laminations,

  Showing that $MCR(\Lambda(X,Y))$ is 
  closed is non-trivial, and may in fact be false. It would follow
  from Remark 2.9 of \cite{HOP21}, that says that the complementary
  components of a maximal chain-recurrent lamination consists of ideal
  triangles, once-punctured monogons, or once-punctured bigons.

  The argument is that Remark 2.9 of \cite{HOP21} tells us that 
  the complementary components of $\lambda$ must also 
  be triangles or once-punctured monogons and bigons, and 
  hence $\lambda$ is maximal. However, a proof of the remark 
  is not furnished anywhere in literature, and it is not clear 
  if it is even true. It is possible that $\lambda_n$ 
  each have a complementary component that is an ideal 
  square, and that in the limit, $\lambda$ 
  has an ideal square as a complementary component, and one 
  of the diagonal leaves can be added to $\lambda$ 
  preserving chain-recurrence.
\end{Remark}

It remains to show that $SV_X(\Lambda(X,Y))$ 
consists of extremal vectors in $\Env_0(X,Y)$. 
This follows immediately from the proof of 
Theorem 5.2 in \cite{Thu86}. We 
provide the main argument, but refer the 
reader to details in \cite{Thu86}. 
We prove Lemma \ref{lem:stretch-vectors-not-internal}:

\begin{Lemma*}
  Let $\Lambda$ be chain-recurrent or empty, 
  and let $\lambda \in M(\Lambda)$, 
  and let $v_{\lambda} \in SV_X(\Lambda)$ be 
  a convex combination:
  $v_{\lambda} = \sum_i a_i v_{\lambda_i}$, where 
  $\lambda_i \in M(\Lambda)$. Then 
  $v_{\lambda_i} = v_{\lambda}$ for all $i$.
\end{Lemma*}

\begin{proof}
  For a lamination $\mu$, denote by $\mu^{CR}$ the maximal
  chain-recurrent sublamination of $\mu$. To prove the lemma, suffices
  to show that $\lambda_i^{CR} = \lambda^{CR}$. 

  Suppose not, and let $\alpha_n$ be a sequence of 
  simple closed multi-curves Hausdorff-converging to 
  $\lambda^{CR}$. Denoting the curve complex 
  of $S$ by $\mathcal{C}(S)$, and fixing $X\in \mc{T}(S)$, 
  we define a map from 
  $T\mc{T}(S) \times \mathcal{C}(S)$ by sending 
  $v,\alpha$ to $D_v \log l_{\alpha}$. 
  This map is linear in the first agument, and is continuous in the
  second argument, where we think of $\mathcal{C}(S)$ as a space
  endowed with the Hausdorff topology on $X$.

  By Thurston's construction of $\mathrm{Stretch}(X,\lambda,t)$, it
  follows that if $\alpha_n$ Hausdorff converge to $\lambda^{CR}$,
  then $\lim_{n\to \infty} D_{v_{\lambda}} \log(l_{\alpha_n}) = 1$. 
  We argue that for any $i$, 
  $\lim_{n\to \infty} D_{v_{\lambda_i}} \log(l_{\alpha_n}) < 1$, 
  and the contradiction follows.

  Indeed, if $\lim_{n\to \infty} D_{v_{\lambda_i}} \log(l_{\alpha_n})
  = c \ge 1$, then there would exist some weights $w_n$ on 
  $\alpha_n$ and some measured lamination 
  $\lambda^m$ whose support contains $\lambda^{CR}$, 
  for which $\alpha_n$ converge to $\lambda^m$ in 
  measure. Moreover, for this measured lamination, 
  we would have $D_{v_{\lambda_i}} \log(l_{\lambda^m}) = c$. 
  This means that $\lambda^{CR}$ is maximally stretched 
  along $v_{\lambda_i}$, a contradiction.
\end{proof}

\subsection{Proving Theorem \ref{thm:stretch-vectors-extremal}}

We now prove Theorem \ref{thm:stretch-vectors-extremal}:

\begin{proof}
  Let $v$ be an extremal point in $\bm{S}_X$. By Remark 
  1.12 and Theorem 1.11 of \cite{PW22}, 
  $v$ must lie in $\Env_0(X, Y)$ 
  for some $Y \in \mc{T}(S)$. In particular, it must be 
  extremal in $\Env_0(X, Y)$, and hence, 
  by Theorem \ref{thm:conv-hull-env0}, $v \in SV_X$. 
  Conversely, by Lemma \ref{lem:stretch-vectors-not-internal}, 
  any stretch vector must be extremal.
\end{proof}

\subsection{Computing Twisting from Shearing}
\label{subsec:twist-compute}

We compute the twist parameters along Thurston geodesics defined by
laminations that are completions of pair of pants decompositions of
$S$. A large portion of this work was independently done in
\cite{HOP21}, but we leave it here for the sake of exposition. 
We begin by reviewing some notation that will help us 
later on.

\subsection{Introduction and Notation}
\label{sec:inf-env-notation}

Let $\mc{C} = \{c_1,\hdots, c_{N}\}$ be a pair-of-pants
decomposition of $S$, and let $\lambda$ be a chain-recurrent
completion of $\mc{C}$ to a triangulation. Let $X\in \mc{T}(S)$, and let
$X_t$ be the Thurston geodesic $\textrm{Stretch}(X,\lambda,t)$.  To
compute the twisting along $X_t$, we must compute 
$\tau_{c_i}(t)$ as a function of
$s_{\lambda}$. This can be done from the definition, but depends on
the topological type of $\lambda$. Let $c\in \mc{C}$ be some curve, and
$\tilde c$ a lift of it. Let $P_1$, $P_2$ be its adjacent
pairs-of-pants, with lifts $\tilde P_i$ adjacent to $\tilde c$. 

Let $p_1(t)$ and $q_2(t)$ be defined as in Section \ref{sec:intro}
be the distinguished points used to compute $s_c(X_t)$ and
$\tau_c(X_t)$, where we choose lifts carefully so that $p_i$ and $q_i$
lie on the same lift $\tilde c$ of $c$.

We define $\Delta_{P_i,c,\lambda}(t)$ as the signed distance from
$q_i(t)$ to $p_i(t)$, where the sign is positive if to get from
$q_i(t)$ to $p_i(t)$, one has to turn left from the perspective of
$\tilde P_i$. The functions $\Delta_{P_i,c,\lambda}(t)$ are actually
intrinsic to $P_i$, and are well-defined up to a choice of the lifts
of all of the curves in $\mc{C}$. In particular,
$\Delta_{P_1,c,\lambda}(t)+\Delta_{P_2,c,\lambda}(t) = s_c(X_t)
- \tau_c(X_t)$, so we devote the rest of this section to computing
$\Delta_{P,c,\lambda}(t)$, where $P$ is a pair of pants in $S\ssm
\mc{C}$, $c\in \mc{C}$ and $\lambda$ is a triangulation of $P$.

\begin{Definition}
  A \textbf{geodesic triangulation} of a hyperbolic pair of pants $P$ 
  is a decomposition of $P$ into two ideal triangles, 
  whose edges are glued together.
\end{Definition}

\begin{Lemma}
  Let $P$ be a hyperbolic pair of pants, then there are 
  exactly $32$ geodesic triangulations of $P$ up to 
  boundary-preserving isomorphism of $P$.
\end{Lemma}

This fact is stated in the beginning of section 3.3 of \cite{PT07}, 
and follows from the discussion in section 2.6 of \cite{HP92}. 

In fact, let $P$ be a pair of pants with 
cuff curves $\gamma_1$, $\gamma_2$, and $\gamma_3$. 
The topological type of a geodesic triangulation of $P$ is 
precisely characterized by the following information
\begin{itemize}
  \item The twist direction at every cuff $\gamma_i$ of $P$ 
    (there are two possibilities per cuff, so eight options in total)
  \item The number of leaves converging to 
    $\gamma_1, \gamma_2, \gamma_3$, with 
    possibilities $(2,2,2), (1,1,4), (1,4,1)$ and $(4,1,1)$.
\end{itemize}

\begin{Definition}
  If $\gamma$ is a cuff curve of $P$, 
  and $\lambda$ is a triangulation of $P$, then we say that 
  \begin{itemize}
    \item $\lambda$ is $3$-symmetric if $\lambda$ has precisely 
      two leaves converging to $\gamma$
    \item $\lambda$ is $2$-symmetric around $\gamma$ 
      if $\lambda$ has precisely four leaves converging 
      to $\gamma$
    \item $\lambda$ is asymmetric around $\gamma$ 
      if $\lambda$ has precisely one leaf converging 
      to $\gamma$
  \end{itemize}
\end{Definition}

\begin{figure}  
  \begin{center}
  \includegraphics[width=0.8\textwidth]{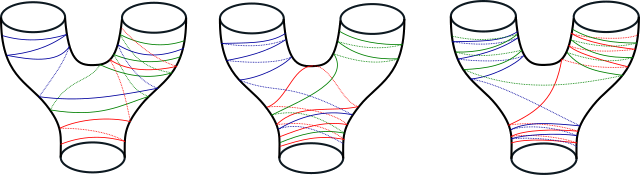}
\end{center}
  \caption{Asymmetric, $2$-symmetric, and $3$-symmetric laminations $\lambda$. The curve $\gamma$ is the bottom one}
  \label{fig:completions}
\end{figure}

We compute $\Delta_{c,P,\lambda}$ in each of these 
cases. We will first fix some notation in all of the following
computations.

Let $\gamma_i$ be the three boundary components of 
$P$, and we will always assume that $c = \gamma_1$. 
Let $\gamma_{ij}$ be the leaf of $\lambda$ which is
asymptotic to both $\gamma_i$ and $\gamma_j$.  For ease of notation
later on, we define $l_i = l_{\gamma_i}(X_0)$, and $s_i
= s_{\gamma_i}(X_0)$, $s_{ji} = s_{ij} = s_{\gamma_{ij}}(X_0)$. Note 
that in the different possibilities for $\lambda$, not all of 
the variables $s_{ij}$ have meaning. For example, 
if $\lambda$ is $2$-symmetric, $\gamma_{23}$ 
is replaced by $\gamma_{22}$.
\subsection{$\lambda$ is $3$-symmetric around $c$}
In this subsection, we compute $\Delta_{P,c,\lambda}$, where 
$\lambda$ is a $3$-symmetric triangulation which twists left 
at every cuff. We will compute $\Delta_{P,c,\lambda}(X_0)$ 
in terms of the shearing co-ordinates of $X_0$. 

We fix an identification of $\tilde S$ with the
upper-half plane such that $\tilde \gamma_1$ is the line
$\{(0,t):t\in\R\}$, oriented so that $P$ lies to the left of this
line.

We know that $\gamma_{13}$ is asymptotic to 
$\gamma_1$, and twists left, meaning that there is a 
lift of it of the form $\{(x,t):t\in \R\}$, for some 
$x<0$. Similarly, $\gamma_{12}$ has a lift of the form 
$\{(y,t):t\in \R\}$, oriented so that $(0,0)$ 
lies to the right of them. 

Up to rescaling the identification 
above, we can assume that $y = x+1$. This will allow 
us to explicitly compute $x$ later on. In our choice 
of specified lifts for the shearing co-ordinates, 
we choose these lifts of $\gamma_{ij}$ with this 
orientation. This choice 
will not matter after differentiating, as a 
different choice of lifts will only change 
the answer by an additive constant. 

The deck action of $[\gamma_1]^{-1}$ on $\HH^2$ 
sends the vertical lines $\{(a,t):t\in \R\}$ to 
$\{(e^{-l_1}a,t):t\in \R\}$, and 
in particular, must send one lift of 
$\gamma_{ij}$ to another one. Drawing two 
lifts of $\gamma_{23}$, one from $(x,0)$ to $(x+1,0)$ 
and one from $(x+1,0)$ to $(e^{-l_1}x,0)$, 
we obtain lifts of the complementary triangles 
of $\lambda$. 
\begin{figure} 
  \begin{center}
  \begin{tikzpicture}[scale=2]
    \def\x{-4}
    \def\y{-3}
    \def\r{(\y-\x)/2} 
    \def\s{(\x/e - \y)/2}

    \draw[thick,->] (0,0) -- (0,3);
    \draw[thick,<->] (0.25,0) -- (-4.3,0);

    \node[ForestGreen] at ({\y},3.2) {$\tilde \gamma_{12}$};
    \node[OrangeRed] at ({\x},3.2) {$\tilde \gamma_{13}$};
    \node at (0,3.2) {$\tilde \gamma_{1}$};

    \draw[thick, ForestGreen] ({\y}, 0) arc (0:180:{1/3});
    \draw[thick, NavyBlue] ({\y}, 0) arc (0:180:{7/24});
    \draw[thick, ForestGreen] ({\y}, 0) arc (0:180:{13/48});
    \draw[thick, OrangeRed] ({\y-2/3}, 0) arc (0:180:{1/6});

    \draw[thick, OrangeRed] ({\x/e}, 0) arc (0:180:{1/2});
    \draw[thick, NavyBlue] ({\x/e}, 0) arc (0:180:{3/8});
    \draw[thick, OrangeRed] ({\x/e}, 0) arc (0:180:{5/16});
    \draw[thick, ForestGreen] ({\x/e-1}, 0) arc (0:180:{(\x/e-1-\y)/2});
    \draw[thick, ->] ({\y}, 0) arc (0:180:{1/4}) node[anchor=north, pos=0.5] {$\gamma_2$};
    \draw[thick, ->] ({\x/e}, 0) arc (0:180:{1/4}) node[anchor=north, pos=0.5] {$\gamma_3$};
    
    \draw[thick, OrangeRed, ->] ({\x},0) -- ({\x},3);
    \draw[thick, ForestGreen,->] ({\y},0) -- ({\y},3);
    \draw[thick, NavyBlue] ({\x}, 0) arc (180:0:{(\y - \x)/2});
    \draw[thick, BurntOrange] ({(\x + \y)/2},{\y - \x}) circle ({(\y-\x)/2});

    \foreach \n in {1,2,3,4}
    {
      \draw[thick, OrangeRed,->] ({\x/(e^\n)},0) -- ({\x/(e^\n)},3);
      \draw[thick,ForestGreen,->] ({\y/(e^\n)},0) -- ({\y/(e^\n)},3);
      \draw[thick,NavyBlue] ({\x/e^\n}, 0) arc (180:0:{(\y/e^\n - \x/e^\n)/2});
      \draw[thick,NavyBlue] ({\y/e^(\n-1)}, 0) arc (180:0:{(\x/e^\n - \y/e^(\n-1))/2});
    \draw[thick,BurntOrange, ->] ({(\x/(e^\n) + \y/(e^\n))/2},{\y/(e^\n) - \x/(e^\n)}) circle ({(\y/(e^\n)-\x/(e^\n))/2});
    \draw[thick,BurntOrange, ->] ({(\x/(e^\n) + \y/(e^(\n-1)))/2},{-\y/(e^(\n-1)) + \x/(e^\n)}) circle ({(\y/(e^(\n-1))-\x/(e^\n))/2});
    }

    \node at ({\x},-0.2) {$x$};
    \node at ({\y},-0.2) {$x+1$};
    \node at ({\x/e},-0.2) {$e^{-l_1}x$};
    \node at ({(\x+\y)/2},{(\y-\x)/2 + 0.2}) {$q$};
    \fill ({\x},0) circle (0.02);
    \fill ({\y},0) circle (0.02);
    \fill ({\x/e},0) circle (0.02);
    \fill ({(\x+\y)/2},{(\y-\x)/2}) circle (0.02);

    \node at ({(\x + \y)/2}, 2.5) {$\Delta_1$};
    \node at ({(\y + \x/e)/2}, 2.5) {$\Delta_2$};

    \draw[thick,|-|] ({\y},{2*\r}) -- ({\y},{2*\s}) node[fill=white, anchor=center, pos=0.5, rotate=90] {$s_{12}$};
    \draw[thick,|-|] ({\x/e},{2*\s}) -- ({\x/e},{1/e}) node[fill=white, anchor=center, pos=0.5,rotate=90] {$s_{13}$};

  \end{tikzpicture}
\end{center}
  \caption{Lifting a $3$-symmetric lamination}
  \label{fig:3-symmetric-lift}
\end{figure}
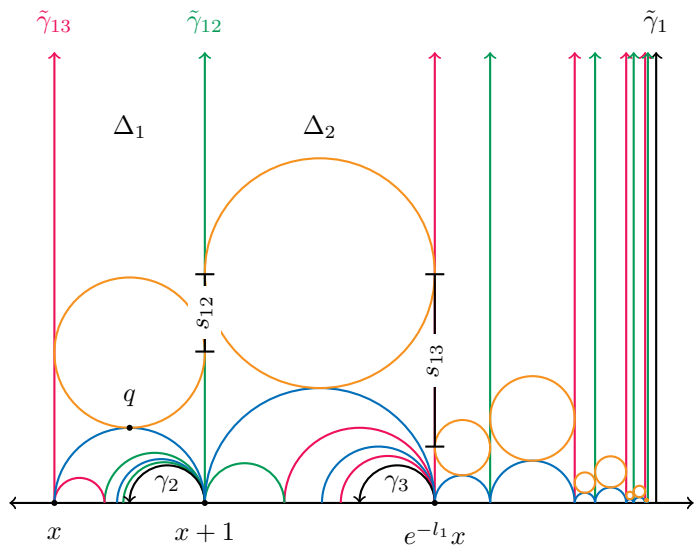

In particular, one can readily see from this picture 
that $s_{12} + s_{13} = -l_1$, and in general, 
$|s_{ij} + s_{jk}| = l_j$, where the sign is 
positive if $\gamma_{ij}$ and $\gamma_{jk}$ twist 
to the right at $\gamma_j$, and is negative 
otherwise. This is a special case of Lemma 3.2 in 
\cite{BBFS09}

By definition, the shearing co-ordinates of $\gamma_{12}$ is the
difference between the medians along $\tilde \gamma_{12}$ of the two
triangles drawn ($\Delta_1$ and $\Delta_2$). This gives us the
(Euclidean) radius of the incircle on the right, which tells us that
$e^{-l_1}x - (x+1) = e^{s_{12}}$.

In particular,
\begin{align*}
  x = \frac{1 + e^{s_{12}}}{e^{-l_1} - 1}
\end{align*}

The next step is to find the endpoints of $\tilde \gamma_2$ 
in terms of $x$. Since $\gamma_{12}$ is 
asymptotic to $\gamma_2$, it follows that one of 
the endpoints of $\gamma_2$ is $x+1$. The 
second endpoint can be computed by taking the 
limit $[\gamma_2]^n \circ x$, where $[\gamma_2]$ 
is the deck transformation corresponding to 
$\gamma_2$. Call this second endpoint $p^*$. We 
now wish to compute $p^*$ in terms of $x$ and the 
shearing co-ordinates.

To do this, we apply the isometry $\vphi(z) = \frac{x-z}{z-(x+1)}$, 
which will send the lift of $\gamma_2$ at 
$[p^*, x+1]$ to an arc $[\vphi(p^*), \infty]$. This 
isometry also sends the arc $[x,x+1]$ to the arc $[0,\infty]$, 
and we can now use shearing co-ordinates to find that its 
first translate under $[\gamma_2]$ must lie on $[w,\infty]$ 
Where $w = e^{s_{23}} + e^{s_{23}}e^{s_{12}}$ (see picture):

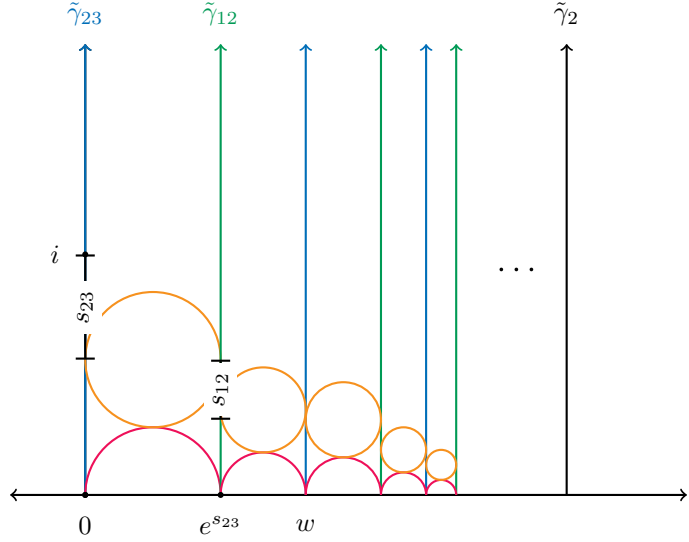
\begin{figure} 
  \begin{center}
  \begin{tikzpicture}[scale=2]
    \def\x{0}
    \def\y{0.9} 
    \def\r{(\y-\x)/2}
    \def\s{(-4/e + 3)/2}
    \def\z{\y*(1 + 2*\s-2*\r)}

    \draw[thick,->] (0,0) -- (0,3);
    \draw[thick,<->] (-0.5,0) -- (4,0);
    \draw[thick, ->] ({3.2},0) -- ({3.2},3);
    \node at (3.2, 3.2) {$\tilde \gamma_{2}$};

    \node[NavyBlue] at ({\x},3.2) {$\tilde \gamma_{23}$};
    \draw[thick, NavyBlue, ->] ({\x},0) -- ({\x},3);
    \draw[thick, OrangeRed] ({\x}, 0) arc (180:0:{(\y - \x)/2});
    \draw[thick, BurntOrange] ({(\x + \y)/2},{\y - \x}) circle ({(\y-\x)/2});
    \node[ForestGreen] at ({\y},3.2) {$\tilde \gamma_{12}$};
    \draw[thick, ForestGreen,->] ({\y},0) -- ({\y},3);

    \draw[thick, NavyBlue, ->] ({\z},0) -- ({\z},3);
    \draw[thick, OrangeRed] ({\y}, 0) arc (180:0:{(\z - \y)/2});
    \draw[thick, BurntOrange] ({(\y + \z)/2},{\z - \y}) circle ({(\z-\y)/2});

    \def\A{\z + 0.5}
    \draw[thick, ForestGreen, ->] ({\A},0) -- ({\A},3);
    \draw[thick, OrangeRed] ({\z}, 0) arc (180:0:{(\A - \z)/2});
    \draw[thick, BurntOrange] ({(\z + \A)/2},{\A - \z}) circle ({(\A-\z)/2});

    \def\B{\A + 0.3}
    \draw[thick, NavyBlue, ->] ({\B},0) -- ({\B},3);
    \draw[thick, OrangeRed] ({\A}, 0) arc (180:0:{0.3/2});
    \draw[thick, BurntOrange] ({(\A + \B)/2},{0.3}) circle ({0.3/2});

    \def\C{\B + 0.2}
    \draw[thick, ForestGreen, ->] ({\C},0) -- ({\C},3);
    \draw[thick, OrangeRed] ({\B}, 0) arc (180:0:{0.2/2});
    \draw[thick, BurntOrange] ({(\B + \C)/2},{0.2}) circle ({0.2/2});

    \draw[thick] ({\C+0.3},1.5) circle ({0.005});
    \draw[thick] ({\C+0.4},1.5) circle ({0.005});
    \draw[thick] ({\C+0.5},{1.5}) circle ({0.005});

    \node at ({\x},-0.2) {$0$};
    \fill ({\x},0) circle (0.02);
    \node at ({\y},-0.2) {$e^{s_{23}}$};
    \fill ({\y},0) circle (0.02);
    \node at ({\z},-0.2) {$w$};
    \fill ({\x},0) circle (0.02);

    \node at ({\x-0.2},1.6) {$i$};
    \fill ({\x},1.6) circle (0.02);

    \draw[thick,|-|] ({\x},{0.9}) -- ({\x},{1.6}) node[fill=white, anchor=center, pos=0.5, rotate=90] {$s_{23}$};
    \draw[thick,|-|] ({\y},{\A-\z}) -- ({\y},{0.9}) node[fill=white, anchor=center, pos=0.5, rotate=90] {$s_{12}$};
    \mute{
    \node at ({(\x + \y)/2}, 2.5) {$\Delta_1$};
    \node at ({(\y + \x/e)/2}, 2.5) {$\Delta_2$};

    \draw[thick,|-|] ({\y},{2*\r}) -- ({\y},{2*\s}) node[fill=white, anchor=center, pos=0.5, rotate=90] {$s_{12}$};
    \draw[thick,|-|] ({\x/e},{2*\s}) -- ({\x/e},{2*\s/e}) node[fill=white, anchor=center, pos=0.5,rotate=90] {$s_{13}$};

    }

  \end{tikzpicture}
\end{center}
\caption{The lift of the $3$-symmetric lamination after applying the isometry $\vphi$}
  \label{fig:3-symmetric-lift2}
\end{figure}

Since the translation length of the aciton of $[\gamma_2]$ 
is precisely $l_2$, it follows that we must have that
$\vphi(p^*) - w = e^{l_2}(\vphi(p^*))$, or, solving for 
$\vphi(p^*)$ and plugging in $\vphi^{-1}(z) = x + \frac{z}{z+1}$ 
yields:
\begin{align*}
  p^* = x + \left(\frac{e^{s_{23}} + e^{-l_2}}{e^{s_{23}}+1}\right)
\end{align*}

All that remains is to compute the intersection point of 
$[0,\infty]$ and the orthogeodesic between $[p^*, x+1]$ and
$[0,\infty]$. In other words, we must find the point 
$p$ as described at the start of this section.

Under our lift, the incircle of $\Delta_1$ intersects 
$\tilde \gamma_{13}$ at $i$, meaning that the horocycle connecting 
this intersection point to $\tilde \gamma_1$ must also 
intersect $\tilde \gamma_1$ at $i$. Thus, 
$\Delta_{P,c,\lambda}(X_0) = \log(-ip)$.

\begin{Claim}
  Let $C_1$ be a circle whose center lies on the $x$-axis intersects the 
  $x$-axis at $(a,0)$ and $(b,0)$ with $0<a<b$. Let $C_2$ 
  be a circle centered at $0$ intersecting $C_1$ at right angles. 
  Then $C_1$ has radius $\sqrt{ab}$
\end{Claim}

\begin{proof}
  The triangle $(0,0), (0,\frac{a+b}{2}), C_1\cap C_2$ is a 
  right triangle with hypothenuse $\frac{a+b}{2}$, and other whose 
  other edges have lengths equal to $\frac{b-a}{2}$ and to 
  the radius of $C_2$. The claim follows.
\end{proof}

By the above claim, it follows that $p = i\sqrt{(x+1)p^*}$, 
and since $q(t)$ is located at $i$ by the choice 
of normalization, we get that:
\begin{align} \label{eq:3-symmetric-delta}
  \Delta_{P,c,\lambda}(X_0)
  = \frac12\log\left((x+1)\left(x+\frac{e^{s_{23}}
  + e^{-l_2}}{e^{s_{23}}+1}\right)\right)
\end{align}
where $x = \frac{1 + e^{s_{12}}}{e^{-l_1} - 1}$, and 
$s_{ij}$ are the shearing co-ordinates of $X_0$, as described 
before. 

Recall that these shearing co-ordinates $\{s_{ij}\}$ 
satisfy the system of equations $s_{ij} + s_{jk} = - l_j$, 
as we assumed that the arcs $\gamma_{ij}$ all twist 
left around the cuffs. In particular, by solving this system 
of equations, we can compute $\Delta_{P,c,\lambda}(X_0)$ 
as an explicit function of $l_{\gamma_1},l_{\gamma_2}$, and 
$l_{\gamma_3}$. 

A similar computation also holds when the twist directions 
of $\lambda$ around the $\gamma_i$'s are arbitrary. However, 
in this case, every appearance of $l_i$ is replaced 
with $-l_i$ if $\lambda$ twists to the right around 
$\gamma_i$, and the twisting is in the opposite direction 
if $\lambda$ twists to thr right around $\gamma_1$. 

For example, if $\lambda$ twists to 
the right around $\gamma_1$, then we choose a normalization 
and lifts of $\gamma_{13}$ and $\gamma_{12}$ 
that are vertical lines intersecting the $x$-axis at $x$ and $x+1$. 
The next lift of $\gamma_{13}$ is then the vertical 
line at $e^{l_1}x$, and not at $e^{-l_1}x$. The 
computation of $x, p^*$, and therefore 
$\Delta_{\mc{P},c,\lambda}$ follow in the same way, 
but with $l_1$ replaced with $-l_1$, and the signed 
distance from $p$ to $q$ picking up a negative sign.

If $\lambda$ twists to the right around $\gamma_2$, 
then in Figure \ref{fig:3-symmetric-lift}, $\tilde \gamma_2$ 
would be drawn starting at $x+1$, and going to the 
right of $\tilde \gamma_{12}$. After passing through 
the same mobius transformation, we arrive at a reflected 
picture in place of Figure \ref{fig:3-symmetric-lift2}, 
where the deck transformation 
pushing lifts of $\gamma_{23}$ towards 
$\tilde \gamma_2$ acts in the opposite direction. 
Thus, we would get $\vphi(p^*)-w = e^{-l_2}\vphi(p^*)$, 
and the rest follows.

Similar computations can be done for all permutations 
of twisting directions of $\lambda$ around the $\gamma_i$'s, 
yielding:
\begin{align}\label{eq:3-symmetric}
  \eps_1\frac12\log\left((x+1)\left(x+\frac{e^{s_{23}}
  + e^{-\eps_2l_2}}{e^{s_{23}}+1}\right)\right)
\end{align}
where 
\begin{itemize}
  \item $\eps_i = 1$ if $\lambda$ twists to the 
    left around $l_i$, and $\eps_i = -1$ otherwise.
  \item $x=x(t) = \frac{1 + e^{s_{12}}}{e^{-\eps_1l_1} - 1}$, 
  \item $s_{ij}$ are the shearing co-ordinates, which satisfy
    $s_{ij} = \frac12(\eps_kl_k - \eps_il_i - \eps_jl_j)$ for $k\neq
    i \neq j \neq k$
\end{itemize}

The formula above for $\Delta_{P,\gamma_1,\lambda}$ 
does not appear to be symmetric under replacing the labels of 
$\gamma_2$ and $\gamma_3$, and indeed 
there is no reason for it to be, as we made the 
choice to compute $p$ using $\gamma_2$. 
However, one may verify that choosing to compute 
$p$ using $\gamma_3$ is analogous to replacing 
$x$ with $x+1$, and replacing $x+1$ with $e^{-l_1}x$, and 
rescaling by a factor of $e^{-l_1}x-(x+1)$. This yields 
a difference of $l_1/2$ in the computation, as expected.

\subsection{$\lambda$ is $2$-symmetric around $c$}
In this section, we compute $\Delta_{P,c,\lambda}$, where 
$\lambda$ is a $2$-symmetric or asymmetric triangulation with 
respect to $c$, which twists left 
at every cuff. We will compute $\Delta_{P,c,\lambda}(X_0)$ 
in terms of the shearing co-ordinates of $X_0$. 

Let $\gamma_i$ be the three boundary components of 
$P$, and label $c = \gamma_1$. Let $\gamma_{ij}$ be the 
leaf of $\lambda$ which is asymptotic to both $\gamma_i$ and
$\gamma_j$. Note that not all combinations of $i$ and $j$ 
yield a viable $\gamma_{ij}$, as in the case 
that $\lambda$ is $2$-symmetric, $\gamma_{23}$ does 
not exist as no leaf is aysmptotic to both $\gamma_2$ 
and $\gamma_3$. In this case,

We define $l_i = l_{\gamma_i}(X_0)$,
and $s_i = s_{\gamma_i}(X_0)$, $s_{ij} = s_{\gamma_{ij}}(X_0)$.

We follow a similar computation as in the section above to get that:

\begin{align}\label{eq:2-symmetric}
  \Delta_{P,c,\lambda}(X_0) = \eps_1\frac12\log\left((x+1)(x+e^{-\eps_2l_2})\right)
\end{align}
where 
\begin{align*}
  x = \frac{1 + e^{s_{12}} + e^{s_{12}+s_{11}}
  + e^{s_{12}+s_{11}+s_{13}}}{e^{-\eps_1l_1} - 1}
\end{align*}
and 
\begin{itemize}
  \item $\eps_i = 1$ if $\lambda$ twists to the 
    left around $l_i$, and $\eps_i = -1$ otherwise.
  \item $s_{ij}$ are the shearing co-ordinates, which satisfy:
    \begin{align*}
      s_{11} &= \frac12(-\eps_1l_1 + \eps_2l_2 + \eps_3l_3)\\
      s_{12} &= -\eps_2 l_2\\
      s_{13} &= -\eps_3 l_3
    \end{align*}
\end{itemize}

\subsection{$\lambda$ is asymmetric around $c$}

In this section, we compute $\Delta_{P,c,\lambda}$, where 
$\lambda$ is a $2$-symmetric triangulation with 
respect to $c$, which twists left 
at every cuff.

Let $\gamma_i$ be the three boundary components of 
$P$, and label $c = \gamma_1$ as before. Let $\gamma_{ij}$ be 
defined as before, and assume that $\lambda$ is 
$2$-symmetric around $\gamma_2$. When computing 
$\Delta_{P,c,\lambda}$, we will use an orthogeodesic 
from $\tilde \gamma_2$ to $\tilde \gamma_1$. We 
define $l_i$, $s_{ij}$ as before, noting that in 
this case, the only co-ordinates are 
$s_{12}$, $s_{23}$, and $s_{22}$. 

We follow a similar computation as in the section above to get that, 

\begin{align} \label{eq:a-symmetric}
  \Delta_{P,c,\lambda}(X_0) = \eps_1\frac12\log\left((x+1)\left(x+
  \frac{e^{s_{22}} + e^{s_{22}+s_{23}} + e^{2s_{22} + s_{23}} + e^{-\eps_2l_2}}
  {e^{s_{22}} + e^{s_{22} + s_{23}} + e^{2s_{22} + s_{23}} + 1}
  \right)\right)
\end{align}
Where
\begin{itemize}
  \item $\eps_i = 1$ if $\lambda$ twists to the 
    left around $l_i$, and $\eps_i = -1$ otherwise.
  \item $x = \frac{1}{e^{-\eps_1l_1} - 1}$
  \item $s_{ij}$ are the shearing co-ordinates, which satisfy:
    \begin{align*}
      s_{22} &= \frac12(-\eps_2 l_2 + \eps_1 l_1 + \eps_3 l_3)\\
      s_{12} &= -\eps_1 l_1\\
      s_{23} &= -\eps_3 l_3
    \end{align*}
\end{itemize}

\subsection{Twist Widths Between Stretch Paths}
We have seen how twisting co-ordinates change along stretch paths
defined by any lamination which is a completion of a pair-of-pants
decomposition of a surface $S$. We now end this subsection 
with an explicit formula for the twisting along 
stretch paths.

Consider the foliation of $\mc{T}(S)$ given by stretch paths
corresponding to $\lambda$. For any $Y\in \mc{T}(S)$, we can define
the \textbf{negative-time stretch path} from $Y$ by setting
$\mathrm{Stretch}(Y, \lambda, -t)$ to be the unique point $X \in
\mc{T}(S)$ such that $\mathrm{Stretch}(X, \lambda, t) = Y$, 
and $d_{Th}(\mathrm{Stretch}(Y,\lambda,-t),Y) = t$. Denote
$Y^{\lambda}_{-t} = \mathrm{Stretch}(Y,\lambda,-t)$. For $X\in
\mc{T}(S)$, we denote $X^{\lambda}_t = \mathrm{Stretch}(X,\lambda,t)$, 

Fix a pair of pants decomposition $\mc{P}$, and some 
curve $c\in \mc{P}$. Let $P_1$ and $P_2$ be the (possibly non-distinct)
pairs-of-pants adjacent to $c$. We employ the 
shorthand: $\Delta_{\lambda}^i = \Delta_{P_i,c,\lambda}$, 
where $\Delta_{P_i,c,\lambda}$ was defined 
in Section \ref{sec:inf-env-notation}

Note that for any $X\in \mc{T}(S)$, we have:
\begin{align} \label{formula:twist-and-shear}
  \tau_c(X) = s_c(X) - \Delta_{\lambda}^1 - \Delta_{\lambda}^2
\end{align}

\begin{Lemma} \label{lem:twist-computed}
  For any $X,Y\in \mc{T}(S)$, we have:
  \begin{align*}
    \tau_c(X_t^{\lambda}) &=  \tau_c(X)e^t +
    \Delta_{\lambda}^1(0)e^t
    +\Delta_{\lambda}^2(0)e^t- \Delta_{\lambda}^1(t)
    -\Delta_{\lambda}^2(t)\\
    \tau_c(Y_{-t}^{\lambda}) &= \tau_c(Y)e^{-t} + 
    \Delta_{\lambda}^1(0)e^{-t}
    +\Delta_{\lambda}^2(0)e^{-t}- \Delta_{\lambda}^1(-t)
    -\Delta_{\lambda}^2(-t)
  \end{align*}

  Where $\Delta_{\lambda}^i(s)$ is $\Delta_{\lambda}^i$, where 
  every shearing co-ordinate is multiplied by a factor 
  of $e^s$ for positive or negative $s$. 
\end{Lemma}

\begin{proof}
  Note that $s_c(X^{\lambda}_t) = s_c(X)e^{t}$, and similarly 
  $s_c(Y^{\lambda}_{-t}) = s_c(Y)e^{-t}$. This 
  is because $c$ is contained in $\lambda$, and in the shearing 
  co-ordinates $(s_{\lambda})$ on Teichm\"{u}ller 
  space, stretch lines are exponential scalings.

  For any $t$, and plugging in $s_c(X_t) = e^t s_c(X)$ to 
  Formula \ref{formula:twist-and-shear}, 
  we have $\tau_c(X_t) = s_c(X)e^t - \Delta_{\lambda}^1(t) - 
  \Delta_{\lambda}^2(t)$. To find $s_c(X)$, 
  we set $t=0$ and rearrange, giving us 
  $s_c(X) = \tau_c(X) + \Delta_{\lambda}^1(0) + \Delta_{\lambda}^2(0)$. 
  We then get:
  \begin{align*}
    \tau_c(X_t^{\lambda}) &= \tau_c(X)e^t+
    \Delta_{\lambda}^1(0)e^t
    +\Delta_{\lambda}^2(0)e^t- \Delta_{\lambda}^1(t)
    -\Delta_{\lambda}^2(t)
  \end{align*}
  A similar computation follows for $Y_{-t}^{\lambda}$.
\end{proof}

If $\nu$ is another completion of
$\mc{P}$, we define the \textbf{twist width} 
$\Delta \tau_c(\lambda,\nu,t)$ by 
$\tau_c(X_t^{\lambda}) - \tau_c(X_t^{\nu})$, and write:

\begin{align} \label{eq:twist-width}
  \Delta \tau_c(\lambda,\nu,t)
  &= e^t\Delta_{\lambda}^1(0) -\Delta_{\lambda}^1(t)\\
  &\quad + e^t\Delta_{\lambda}^2(0) -\Delta_{\lambda}^2(t)\\
  &\quad - \left(e^t\Delta_{\nu}^1(0) -\Delta_{\nu}^1(t)\right)\\
  &\quad - \left(e^t\Delta_{\nu}^2(0) -\Delta_{\nu}^2(t)\right)
\end{align}

In order to estimate the twist width, it suffices to estimate each of
the terms above separately.

Observe that by the computations 
in the previous section, $\Delta_{\lambda}^i(t)$ is 
always of the form $\frac12\log(R_{\mc{P},\lambda^i}(t))$, 
where $R_{\mc{P},\lambda}^i(t)$ is some rational polynomial 
over the set $\{e^{l_{\alpha}(t)}\}_{\alpha\in \mc{P}}$. 
Thus, if $\alpha\in \mc{P}$, then we have that 
$l_{\alpha}(t) = l_{\alpha}(X_0)e^t$, and so, for any real $k$, 
we compute:
\begin{align*}
  &= \frac12 e^t \log(e^{kl_{\alpha}(X_0)}R_{\mc{P},\lambda}^i(0))
  - \frac12 \log(e^{kl_{\alpha}(t)}R_{\mc{P},\lambda}^i(t))\\
  &= \frac12 e^t\left(kl_{\alpha}(X_0) + \log(R_{\mc{P},\lambda}^i(0))\right)
  - \frac12 \left(kl_{\alpha}(t) + \log(R_{\mc{P},\lambda}^i(t))\right)\\
  &= \frac12 e^t \log(R_{\mc{P},\lambda}^i(0))
  - \frac12 e^t \log(R_{\mc{P},\lambda}^i(0))\\
  &=e^t\Delta_{\lambda}^i(0) - 
  \Delta_{\lambda}^i(t)
\end{align*}
Where $t$ can also hold negative values. Thus, we are left with:
\begin{Fact} \label{fact:factoring-lengths-out}
  When computing twist widths, we are free to mutliply and divide
  $R_{\mc{P},\lambda}^i(t)$ by any factor of the form
  $e^{l_{\alpha}(t)}$, where $\alpha\in \mc{P}$. 
\end{Fact}

\section{Bounded Width}
\label{sec:bounded-width}

In this section, we prove Theorem \ref{thm:bounded-envelope}. 
To prove this theorem, we first use Theorem \ref{thm:stretch-vectors-extremal} 
to show that it suffices to compute the distance between 
points in the envelope that lie on stretch paths. We then 
examine the lengths of simple closed curves 
in the maximally-stretched lamination. We use twisting 
computations from subection \ref{subsec:twist-compute} 
to find upper and lower boundes of the twisting around it. 
When the curve is long, we 
use an earthquake around it to bound the 
distance in the envelope. When the curve is short, we 
resort to other estimates of the Thurston distance.

Throughout this section, we will use coarse estimate 
notation as in \cite{DLRT20}. We introduce it here.

\begin{Definition}
  Given two quantities (or functions), $A$ and $B$, 
  we write $A \pl B$ if there exists a constant $C$ 
  independent of $A$ and $B$ such that $A \le B+C$. 

  Similarly, we write $A \tl B$ if there exists a constant $C$ 
  independent of $A$ and $B$ such that $A \le BC$. 
  We write $A \pe B$ (resp. $A \te B$) if 
  $A \pl B$ and $B \pl A$ (resp. $A \tl B$ and 
  $B \tl A$)
\end{Definition}

\begin{Remark}
  The relations $\tl, \te, \pl,$ and $\pe$ \textit{do not} 
  behave as regular equalities and inequalities. For example, 
  if $A \pe B$ and $C \pe D$, it is not necessarily 
  true that $\frac{A}{C} \pe \frac{B}{D}$. 
\end{Remark}

\subsection{Twisting and Earthquakes}
In this section, we prove the following, which is of interest 
of its own right:

\begin{Proposition*}
  Let $\alpha$ be a simple closed curve on $S$, and let 
  $X \in \mc{T}(S)$. Then:
  \begin{align*}
    d_{Th}(X,Eq_{\alpha,t}(X)) \pl \log(e^{l_{\alpha}/2}t)
  \end{align*}
\end{Proposition*}

The following lemma will be useful in the proof of
Proposition \ref{prop:earthquake-bound}:

\begin{Lemma} \label{lem:thick-part-intersections}
  Let $\eps > 0$, and $X\in \mc{T}(S)$ be given. Let $\alpha,\beta$ be
  geodesic arc segments, and assume that $\beta$ is  contained in the
  $\eps$-thick part of $X$, $X_{\ge \eps}$. Then 
  \begin{align*}
    i(\alpha,\beta) \le \frac{4l_{\alpha}(X)l_{\beta}(X)}{\eps^2}
  \end{align*}
\end{Lemma}

Lemma \ref{lem:thick-part-intersections} has appeared in literature in
the setting of flat structures \cite{Raf07}, and in various other
contexts for extremal lengths \cite{Min93}. We prove it here in the
setting of hyperbolic lengths and geometric intersection numbers. In
particular, we get:

\begin{Corollary}
  If $\alpha,\beta$ are any closed curves contained in $X_{\ge \eps}$,
  then $i(\alpha,\beta) \tl_{\eps} l_{\alpha}(X)l_{\beta}(X)$
\end{Corollary}

We prove Lemma \ref{lem:thick-part-intersections}:

\begin{proof}
  Let $w$ be the subarc of $\alpha$ of length $\eps/2$ which 
  maximizes $|w\cap \beta|$. Since $\beta$ is contained 
  $X_{\ge \eps}$, it follows that every 
  arc segment of $\beta \ssm (\beta \cap w)$ must be of 
  length at least $\eps/2$, otherwise we could find an essential loop of 
  length $\eps$ in $X_{\ge \eps}$ ($\alpha$ and $\beta$ are geodesic
  segments and hence in minimal position). Denote $|w\cap \beta| = N$, 
  and by the above observation, we get that $N \le \frac{l_{\beta}}{\eps/2}$. 
  Additionally, since $w$ was chosen to maximize $|w\cap \beta|$, it follows 
  that $\frac{l_{\alpha}}{\eps/2} \le N$, and the proof 
  follows by stringing together these inequalities.
\end{proof}

To prove Proposition \ref{prop:earthquake-bound}, we will use the coarse
estimates of the Thurston metric from \cite{LRT14} using a short
marking:

\begin{Definition} \label{def:short-marking}
  Let $X\in \mc{T}(S_{g,n})$. A \textbf{short marking} on $X$ is
  a collection of $N = 3g-g + n$ simple closed curves, $\{\beta_i,
  \beta_i'\}_{i=1}^{N}$ obtained as follows:
  \begin{itemize}
    \item Define $\beta_1$ to be the shortest curve on $X$. Then 
      inducively, define $\beta_i$ to be the shortest curve on $X$ 
      disjoint from $\beta_j$ for all $j < i$.
    \item Define $\beta_i'$ to be the shortest curve which 
      intersects $\beta_i$ transversely.
  \end{itemize}
  We call $\beta_i$ and $\beta_i'$ \textbf{duals} to each other, and 
  we denote $\bar \beta_i = \beta_i'$ and $\bar\beta_i' = \beta_i$.
\end{Definition}

We can estimate intersection numbers of curves with a short marking 
in a similar way to Proposition 3.1 of \cite{LRT14}. By gaining 
control of the number of intersections between a curve $\alpha$ and the 
curves in the short marking, we can gain control 
of the length of each curve in the short marking after earthquaking 
along $\alpha$. We prove:

\begin{Lemma}
  Let $\mu_X$ be a short marking on $X$, and let $\alpha$ be a simple 
  closed curve. For any $\beta\in \mu_X$, we have:
  \begin{align*}
    i(\alpha,\beta)/l_{\beta}(X) \tl e^{l_{\alpha}/2}
  \end{align*}
\end{Lemma}

\begin{proof}
  We split the proof into two cases, depending on the length 
  of $\beta$. We denote $l_{\beta} = l_{\beta}(X)$, and 
  always assume that $i(\alpha,\beta) \ge 1$, otherwise 
  the lemma follows trivially.

  \noindent \textbf{$l_{\beta} < \frac1e$}: In this case, $\beta$ 
  has a collar of radius $2\sinh^{-1}(1/\sinh(l_{\beta}/2))$. By 
  estimating $\sinh(x) \le 2x$ for $x < 1$, and applying the 
  inequality $\sinh^{-1}(x) \ge \log(x)$, we 
  obtain a lower-bound on the width of the collar around $\beta$: 
  $2\log(1/l_{\beta})$.

  For every intersection of $\alpha$ and $\beta$, $\alpha$ 
  must enter and exit the collar of $\beta$, meaning that $i(\alpha,\beta) \le \frac{l_{\alpha}}{2\log(1/l_{\beta})}$
  Since $\alpha$ intersects $\beta$, it follows in particular that 
  we must have $l_{\alpha} \ge 2\log(1/l_{\beta})$, or, 
  $l_{\beta} \ge e^{-l_{\alpha}/2}$. Thus,
  \begin{align*}
    \frac{i(\alpha,\beta)}{l_{\beta}} \le \frac{l_{\alpha}}{l_{\beta}2\log(1/l_{\beta})}
  \end{align*}
  The function $x\log(1/x)$ is increasing on $[0,\frac1e]$, meaning that 
  $2l_{\beta}\log(1/l_{\beta}) \ge e^{-l_{\alpha}/2}l_{\alpha}l_{\alpha}$, and 
  plugging this in, we get the appropriate estimate.

  \noindent \textbf{$\frac1e \le l_{\beta}$}: By Propositon 3.1 of
  \cite{LRT14}, there exists some $C$ such that $l_{\alpha} \ge
  C\sum_{\beta\in \mu_X} i(\alpha,\beta)l_{\bar\beta}$. In particular,
  we get $i(\alpha,\beta) \le C\frac{l_{\alpha}}{l_{\bar\beta}}$, so
  $\frac{i(\alpha,\beta)}{l_{\beta}} \le
  \frac{l_{\alpha}}{l_{\beta}l_{\bar\beta}}$. 

  Let $k(S) < \frac1e$ be such that if $l_{\bar\beta} \le k(S)$, then $\bar\beta$ has a 
  collar of width at least $B(S)$, where $B(S)$ is the Bers constant of $S$. In particular, 
  if $l_{\bar\beta} \le k(S)$, then this means that $\bar\beta$ is a curve in the short 
  marking that was a part of a short pair-of-pants decomposition. 

  If $l_{\bar\beta} \ge k(S)$, then $\frac{i(\alpha,\beta)}{l_{\beta}} \le C\frac{e}{K(S)} l_{\alpha} 
  \tl e^{l_{\alpha}/2}$.

  We now assume that $\frac1e \le l_{\beta}$ and that $l_{\bar\beta} \le k(S)$. In this case, 
  $\bar\beta$ is part of a short pair-of-pants decomposition, and 
  we can use the collar estimates from before to get that $\bar\beta$ has a collar 
  of width at least $2\log(1/l_{\bar\beta})$, so $l_{\beta} \ge 2\log(1/l_{\bar\beta})$. 
  If $\alpha$ intersects $\bar\beta$, then we must have $l_{\alpha} \ge 2\log(1/l_{\bar\beta})$, 
  and, in a similar fashion to the case where $l_{\beta} < \frac1e$, we get:
  \begin{align*}
    \frac{i(\alpha,\beta)}{l_{\beta}} \le \frac{l_{\alpha}}{l_{\beta}l_{\bar\beta}} 
    \le \frac{l_{\alpha}}{l_{\bar\beta}2\log(1/l_{\bar\beta})} \le e^{l_{\alpha}/2}
  \end{align*}

  We are left with the case that $\alpha$ does not intersect $\bar\beta$. Let 
  $U(\bar\beta)$ be a collar around $\bar\beta$ whose 
  boundary components have length at least $\frac1e$. By construction 
  of the dual, $\beta$ must be entirely contained in $X_{\ge \frac1e} \cup U(\bar\beta)$. Let 
  $\beta_1 = \beta \cap X_{\ge \frac1e}$. Note that since $\alpha$ does not intersect 
  $\bar\beta$, it follows that $\alpha$ is disjoint from $U(\bar\beta)$, and 
  so $i(\alpha,\beta) = i(\alpha,\beta_1)$. By Lemma \ref{lem:thick-part-intersections}, 
  $i(\alpha,\beta_1) \tl l_{\alpha}l_{\beta_1} \tl l_{\alpha}l_{\beta}$. Thus, 
  $\frac{i(\alpha,\beta)}{l_{\beta}} \tl e^{l_{\alpha}/2}$.
\end{proof}

\begin{Theorem} \label{thm:short-marking}(Theorem E in \cite{LRT14})
  For any $X,Y\in \mc{T}(S_{g,n})$, we have:
  \begin{align*}
    d_{Th}(X,Y) \pe \max_{\beta\in \mu_X} \log \frac{l_{\beta}(Y)}{l_{\beta}(X)}
  \end{align*}
\end{Theorem}

We now prove Proposition \ref{prop:earthquake-bound}:

\begin{proof}
  Let $\mu_X$ be a short marking on $X$. For 
  any $\beta \in \mu_X$, if $\beta$ is disjoint from $\alpha$, then 
  $l_{\beta}(Eq_{\alpha,t}(X)) = l_{\beta}(X)$, and so 
  $\log \frac{l_{\beta}(Eq_{\alpha,t}(X))}{l_{\beta}(X)} = 0$. Otherwise, 
  we can bound the length $l_{\beta}(Eq_{\alpha,t}(X))$ by:
  \begin{align*}
    l_{\beta}(Eq_{\alpha,t}(X)) \le l_{\beta} + t i(\alpha,\beta)
  \end{align*}
  In particular, by Theorem \ref{thm:short-marking},
  \begin{align*}
    d_{Th}(X,Eq_{\alpha,t}(X)) \pe \max_{\beta\in \mu_X} \log \left(1 + t\frac{i(\alpha,\beta)}{l_{\beta}(X)}\right)
  \end{align*}
  Applying Proposition \ref{prop:earthquake-bound}, we get that
  $\frac{i(\alpha,\beta)}{l_{\beta}} \tl e^{l_{\alpha}/2}$, meaning
  that $d_{Th}(X,Eq_{\alpha,t}(X)) \pl \log(e^{l_{\alpha}/2}t)$, as desired.
\end{proof}

\subsection[Bounded Envelopes]{Bounded Envelopes in $\mc{T}(S)$}
\label{subsec:bounded-envelope}
In this section, we prove that geodesic envelopes have uniformly 
bounded width in $\mc{T}(S_{1,1})$ 
and in $\mc{T}(S_{0,4})$. We first show:

\begin{Lemma} \label{lem:uniform-bound-for-scc}
  Let $S$ be the once-punctured torus or the four-times 
  punctured sphere. There exists some $B = B(S)\in \R$ such that for
  any $X,Y \in \mc{T}(S)$, we have that if $\Lambda(X,Y)$ is a simple
  closed curve, then $w(X,Y) < B$.
\end{Lemma}

From the lemma, we immediately obtain:
\begin{Theorem*}
  Let $S$ be the once-punctured torus or the four-times punctured
  sphere. There exists some $B\in \R$ such that for any $X,Y \in
  \mc{T}(S)$, $w(X,Y) < B$.
\end{Theorem*}

\begin{proof}
  Let $X,Y \in \mc{T}(S)$ be arbitrary, and let 
  $\Lambda(X,Y)$ be the maximally-stretched lamination 
  from $X$ to $Y$. For the four-times punctured sphere, 
  by Proposition 26 of 
  \cite{BZ05}, $\Lambda(X,Y)$ is either maximal, 
  in which case there is a unique geodesic from 
  $X$ to $Y$, or $\Lambda(X,Y)$ 
  contains a simple closed curve, $\alpha$. 
  For the once-punctured torus, the same result 
  holds by Theorem 1.1 of \cite{DLRT20}. Now, 
  notice that geodesic from $X$ to $Y$ must be contained 
  in $\Outenv(X,\alpha) \cap In(Y,\alpha)$, which, by 
  Lemma \ref{lem:uniform-bound-for-scc}, has width bounded by $B$.
\end{proof}

The rest of this section is dedicated to proving
Lemma \ref{lem:uniform-bound-for-scc}. We first fix some 
notation. Fix $Y\in \mc{T}(S)$ 
and a simple closed curve $\alpha$.
\begin{itemize}
  \item Denote by $\lambda^{L}$ and $\lambda^R$ be the two
    maximal chain-recurrent laminations containing $\alpha$, twisting
    to the left and right (respectively) around $\alpha$ 
    (see Figure \ref{fig:lambda-RL})
  \item If $S = S_{1,1}$, define $l_0 = l_{\alpha}(Y)/2$, 
    if $S = S_{0,4}$, define $l_0 = l_{\alpha}(Y)/4$. 
    Define $l_s = l_0e^s$ for any real $s$.
  \item For $Y\in \mc{T}(S)$, 
    we define $Y_{-t}^{L}$ as the unique point in $\mc{T}(S_{0,4})$ 
    such that $\Lambda(Y_{-t}^{L},Y) = \lambda^L$, 
    and $d_{Th}(Y_{-t}^L,Y) = t$. Define 
    $Y_{-t}^R$ in the same manner. 
  \item Let $P_1, P_2$ be the two (possibly non-distinct) pairs of
    pants on opposite sides of $\alpha$. We define 
    $\Delta_i^L(t) = \Delta_{P_i,\alpha,\lambda^L}(t)$ 
    and $\Delta_i^R(t) = \Delta_{P_i,\alpha,\lambda^R}(t)$ 
    (see Section \ref{sec:inf-env-notation})
\end{itemize}

\begin{figure}[h!]
  \begin{center}
    \includegraphics[width=0.35\textwidth]{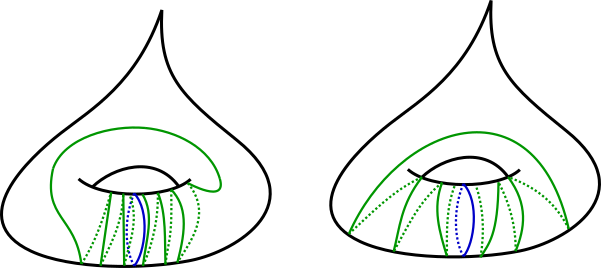}
  \end{center}
  \caption{Definition of $\lambda^R$ (right) and $\lambda^L$ (left). 
  The blue curve is $\alpha$}\label{fig:lambda-RL}
\end{figure}

Employing the above notation, and following the 
computations in Lemma \ref{lem:twist-computed}, we 
compute the relative twisting 
$\mathrm{Twist}_{\alpha}(Y_{-t}^L,Y_{-t}^R)$ 
in either case of $S$:

\begin{Lemma}
  Let $\alpha,l_0,l_s,Y_{-t}^{\pm}\in \mc{T}(S)$ be defined as above,
  then:
  \begin{align}\label{eq:four-punctured-twist-width}
    \mathrm{Twist}_{\alpha}(Y_{-t}^{+},Y_{-t}^{-}) 
    = 4e^{-t}\log\coth(l_0/2)-
    4\log\coth\left(\frac{e^{-t}l_0}{2}\right)
  \end{align}
\end{Lemma}

\begin{proof}
  We first treat the case where $S = S_{0,4}$. In this 
  case, we use the computations from 
  Lemma \ref{lem:twist-computed} to get:
  \begin{align*}
    \mathrm{Twist}_{\alpha}(Y_{-t}^L, Y_{-t}^R) 
    = e^{-t}(\Delta_1^L(0) + \Delta_2^L(0))
    + e^{-t}(\Delta_1^R(0) + \Delta_2^R(0))\\
    - (\Delta_1^L(-t) + \Delta_2^L(-t))
    - (\Delta_1^R(-t) + \Delta_2^R(-t))
  \end{align*}

  Using Example \ref{eq:2-symmetric} with $l_2 = l_3 = 0$, 
  and $s_{11} = -2l_s$, we get 
  $x = 2\frac{1+e^{-l_s}}{e^{-2l_s}-1}$, and so:
  \begin{align*}
    \Delta_1^L = \Delta_2^L
    &= \frac12\log((x+1)^2)\\
    &= \log\left(\frac{e^{-2l_s}+1+2e^{-l_s}}{1-e^{-2l_s}}\right)\\
    &= \log\coth(l_s)
  \end{align*}
  where in the last line we used Fact \ref{fact:factoring-lengths-out}. 
  Similarly, 
  and $\Delta_1^R = \Delta_2^R = 
  -\log\left(\coth(l_s\right)$, 
  and the result follows.

  Next, we consider the case where $S = S_{1,1}$. 
  Using Example \ref{eq:3-symmetric} 
  with $l_2 = l_1 = 2l_s$ and $l_3 = 0$, and 
  getting $s_{12} = -2l_s$, and $s_{23} = \frac12(l_1-l_2-l_3) = 0$.
  This gives $x = \frac{1+e^{-2l_s}}{e^{-2l_s}-1}$, 
  and we compute:
  \begin{align*}
    x+1 &= \frac{2e^{-2l_s}}{e^{-2l_s}-1}\\
    x+\frac{e^{s_{23}}+e^{-2l_s}}{e^{s_{23}}+1}
    &= \frac12 \frac{(1+e^{-2l_s})^2}{e^{-2l_s}-1}
  \end{align*}
  Thus, after applying Fact \ref{fact:factoring-lengths-out}, 
  we get:
  \begin{align*}
    \Delta_1^L = \log \left(\coth(l_s)\right)
  \end{align*}
  $\Delta_2^L$, and $\Delta_i^R$ are similarly computed, 
  giving the result.
\end{proof}

\begin{Remark}
  Parts of the above lemma also follows from formulas 18 and 19 in
  \cite{DLRT20}, with slight modifications to treat negative $t$.
  However, it is a reaffirming sanity check that the computations
  preformed in Section \ref{sec:infinitesimal-envelope} agree with those
  done in \cite{DLRT20}
\end{Remark}

\begin{Lemma} \label{lem:systole-lemma}
  Let $S$ be the once-punctured torus or 
  the four-times punctured sphere, and let 
  $\alpha$ be a simple closed curve in $S$. 
  Then there exists a constant $\eps_0$ 
  such that for any $X \in \mc{T}(S)$, 
  if $l_{\alpha}(X) < \eps_0$, then 
  $\alpha$ is the systole of $X$.
\end{Lemma}

\begin{proof}
  Any pair of distinct simple closed curves 
  on $S$ must intersect. Therefore, if 
  $\alpha$ is sufficiently short, then 
  by the collar lemma, any other simple closed 
  curve in $S$ must be at least the length 
  of the Bers constant of $S$. In particular, 
  this implies that $\alpha$ is the shortest 
  curve on $S$.
\end{proof}

We henceforth write $\eps_S$ to be the constant 
from Lemma \ref{lem:systole-lemma}.

\begin{Lemma} \label{lem:dual-curve-length-ratio}
  Let $S = S_{1,1}$ or $S = S_{0,4}$, 
  and let $X\in \mc{T}(S)$. Let 
  $Y \in \mc{T}(S)$ be obtained from $X$ by 
  twisting along a simple closed curve $\alpha$
  Let $\beta = \bar\alpha$ be the dual curve 
  of $\alpha$ as in Definition \ref{def:short-marking}. 
  If $l_{\alpha}(X) < \eps_S$, then 
\end{Lemma}

\subsection{The Thin Part of $\mc{T}(S)$}

In this subsection, we show that 
$d_{Th}(Y_{-t}^L, Y_{-t}^R)$ is uniformly 
bounded if $l_0e^{-t}$ is small. 

\begin{Lemma}
  Let $\eps = \min(\eps_{S_{0,4}},\eps_{S_{1,1}},\log(2))$, and assume
  that $l_0 \le \eps$. Furthermore, 
  assume that $Y\in \mc{T}(S)$ for $S = S_{1,1}$ or 
  $S = S_{0,4}$. Then $d_{Th}(Y_{-t}^L, Y_{-t}^R)$ is bounded
  uniformly.
\end{Lemma}

\begin{proof}
  Using Theorem E of \cite{DLRT20}, it suffices to check that
  $\frac{l_{\beta}(Y_{-t}^R)}{l_{\beta}(Y_{-t}^L)}$ is uniformly
  bounded, where $\beta$ is the dual curve to $\alpha$. Recall that
  $l_{-t} = l_0e^{-t} \pe l_{\alpha}(Y_{-t}^L) = l_{\alpha}(Y_{-t}^R)$.
  By the collar lemma, and Proposition 3.1 of \cite{LRT14}, we know
  that $l_{\beta}(Y_{-t}^L) \te \log(1/l_{-t}) \tg \log(1/\eps)$.
  Since $Y_{-t}^R$ is obtained from $Y_{-t}^L$ by twisting along
  $\alpha$ for time $\mathrm{Twist}_{\alpha}(Y_{-t}^L, Y_{-t}^R)$, we
  can estimate the length ratio of $\beta$ by:

  \begin{align*}
    \frac{l_{\beta}(Y_{-t}^R)}{l_{\beta}(Y_{-t}^L)}
    &\le \frac{l_{\beta}(Y_{-t}^L) + l_{\alpha}(Y_{-t}^L)
    |\mathrm{Twist}_{\alpha}(Y_{-t}^R,Y_{-t}^L)|}{l_{\beta}(Y_{-t}^L)}\\
    &\tl 1+ \frac{l_0e^{-t}}{\log(1/\eps)} 4\left(e^{-t}
    \log\coth(l_0) + \log\coth(l_0e^{-t})\right)
  \end{align*}
  Where we used the fact that $l_{\alpha}(Y_{-t}^L) 
  \le 4 l_0e^s$.

  We remark that by our choice of $\eps$, we have that
  $\frac{1}{l_0e^{-t}} \ge \log\coth(l_0e^{-t}) > 0$, meaning that:

  \begin{align*}
    \frac{l_{\beta}(Y_{-t}^R)}{l_{\beta}(Y_{-t}^L)}
    &\tl 1+ 4\frac{l_0e^{-t}}{\log(1/\eps)}
    \left(\frac{e^{-t}}{l_0} + \frac{e^{t}}{l_0}\right)
    =1 + \frac{1}{\log(1/\eps)} (e^{-2t} + 1)
  \end{align*}
  Which is uniformly bounded in $l_0$ and in $t$.
\end{proof}

\begin{Lemma}
  Assume that $\eps \le l_0\le 2$, and that $l_0e^{-t} < \eps$ from
  the previous lemma. Then $d_{Th}(Y_{-t}^L, Y_{-t}^R)$ is bounded
  uniformly.
\end{Lemma}

\begin{proof}
  As in the previous lemma, we must estimate
  \begin{align*}
    \frac{l_{\beta}(Y_{-t}^R)}{l_{\beta}(Y_{-t}^L)}
    &\tl 1+ \frac{l_0e^{-t}}{\log(1/\eps)} 4\left(e^{-t}
    \log\coth(l_0) + \log\coth(l_0e^{-t})\right)
  \end{align*}
  The $\log\coth(l_0)$ term is bounded uniformly, 
  so it suffices to observe that 
  $l_0e^{-t} \log\coth(l_0e^{-t})$ is 
  bounded uniformly.
\end{proof}

\begin{Lemma}
  Assume that $2 \ge l_0$, and that $l_0e^{-t} < \eps$ from the
  previous lemmas. Then $d_{Th}(Y_{-t}^L, Y_{-t}^R)$ is bounded
  uniformly.
\end{Lemma}

\begin{proof}
  Let $t_0 = \log(l_0)$, and set $s = t-t_0$. Note 
  that $t_0 \ge \log(2) > 0$ and that $s\in (-\log(\eps),\infty)$ 
  is bounded from below. With this notation, 
  and the same argumens as before, we have:
  \begin{align*}
    \frac{l_{\beta}(Y_{-t}^R)}{l_{\beta}(Y_{-t}^L)}
    &\tl 1+ 4\frac{e^{-s}}{s}
    \|\left(e^{-s-t_0} \log\coth(e^{t_0})
    +\log\coth(e^{-s})\right)\|
  \end{align*}

  As $t_0$ ranges from $0$ to $\infty$, 
  we have that $\frac{1}{e^{t_0}} > \log\coth(e^{t_0})> 0$, 
  and similarly $\frac{1}{e^{-s}} > \log\coth(e^{-s}) > 0$,
  giving:
  \begin{align*}
    \frac{l_{\beta}(Y_{-t}^R)}{l_{\beta}(Y_{-t}^L)}
    &\tl 1+ 4\frac{e^{-s}}{s}\left(e^{-s}+e^s)\right)
  \end{align*}
  Which is bounded uniformly in $s$, since 
  $s$ is bounded from below by $-\log(\eps)$.
\end{proof}

Together, the previous three lemmas tell us:

\begin{Lemma} \label{lem:thin-part-uniform-bound}
  Let $\eps = \min(\eps_{S_{1,1}},\eps_{S_{0,4}},\log(2))$, 
  and assume that $l_0e^{-t} \le \eps$. Then 
  $d_{Th}(Y_{-t}^L, Y_{-t}^R)$ is bounded uniformly.
\end{Lemma}

\subsection{The Thick Part of $\mc{T}(S)$} \label{subsec:thick-part}
In this subsection, we show that 
$d_{Th}(Y_{-t}^L, Y_{-t}^R)$ is uniformly bounded 
when $l_0e^{-t}$ is large. 

\begin{Lemma} \label{lem:thick-part-uniform-bound}
  In the setting of the above lemmas, if $1 < l_0e^{-t}$, then
  $d_{Th}(Y_{-t}^L, Y_{-t}^R)$ is uniformly bounded.
\end{Lemma}

\begin{proof}
  We first treat the case when $S = S_{1,1}$ 
  is the once-punctured torus. We recall 
  that $l_{\alpha}(Y_{-t}^L) = l_{\alpha}(Y_{-t}^R) 
  = 2 l_0$, and by Proposition \ref{prop:earthquake-bound}, 
  it suffices to verify that for $l_0e^{-t} >1$, we 
  have that 
  \begin{align*}
    e^{l_0e^{-t}} \mathrm{Twist}_{\alpha}(Y_{-t}^L,Y_{-t}^R)
    &= 4e^{l_0e^{-t}} (e^{-t}\log\coth(l_0)
    -\log\coth(l_0e^{-t}))\\
    &\le 4e^{l_0e^{-t}} e^{-t}\log\coth(l_0)
    +4e^{l_0e^{-t}} \log\coth(l_0e^{-t})
  \end{align*}
  is bounded uniformly. Note that 
  $1 < \coth(l_0) < \coth(l_0e^{-t})$, 
  so it suffices to observe that $e^{2l_0e^{-t}}\log\coth(l_0e^{-t})$ 
  is bounded uniformly in $t$ and $l_0$, as long 
  as $l_0 e^{-t} > 1$.

  The case when $S = S_{0,4}$ is similar; we must verify 
  that $e^{2l_0e^{-t}}
  \mathrm{Twist}_{\alpha}(Y_{-t}^L,Y_{-t}^R)$ is uniformly bounded.
\end{proof}

\begin{Remark}
  It is a striking coincidence that the twisting 
  along a simple closed curve is decaying 
  at least as fast as $e^{-2l_0e^{-t}}$ in the 
  case of the once-punctured torus and the four-times punctured 
  sphere. In higher complexity surfaces, this is not 
  generally the case. In fact, the factor of $\frac12$ is 
  crucial in Proposition \ref{prop:earthquake-bound}, 
  since $e^{(2+\eps)l_0e^{-t}}\log\coth(l_0e^{-t})$ 
  is not uniformly bounded for any $\eps > 0$.
\end{Remark}

In fact, similar computations can be performed to estimate 
$d_{Th}(Y_{-t}^R, Y_{-t}^L)$, giving:

\begin{Lemma} \label{lem:stretch-paths-uniform-bound}
  $d_{Th}(Y_{-t}^L, Y_{-t}^R)$ and $d_{Th}(Y_{-t}^R, Y_{-t}^L)$ are
  uniformly bounded.
\end{Lemma}

\begin{proof}
  We've seen the cases for when $l_0e^{-t} < \eps$ 
  and when $l_0e^{-t} >1$. 
  then let $Y_0^R$ (resp. $Y_0^L$) be the point on $Y_{-t}^R$
  (resp. $Y_0^L$) for which $l_{\alpha}(Y_0^L) = 1$. 

  Let $t$ be such that 
  $\eps \le l_0e^{-t} \le 1$, by the triangle inequality, 
  $d_{Th}(Y_{-t}^R,Y_{-t}^L) 
  \le 2(1-\eps) + d_{Th}(Y_0^R,Y_0^L)$ 
  which is uniformly bounded. A similar 
  argument can be shown for 
  $d_{Th}(Y_{-t}^L,Y_{-t}^R)$
\end{proof}

\subsection{Geodesic Envelopes}

Before proving Lemma \ref{lem:uniform-bound-for-scc}, we 
first analyze the large-scale structure of 
$Env(X,Y)$ for $X, Y \in \mc{T}(S)$, where $S = S_{1,1}$ 
or $S = S_{0,4}$. Let $\alpha$ be a 
simple closed curve in $S$, and let 
$Y, Y_{-t}^L$ and $Y_{-t}^R$ be defined as before

\begin{Lemma}
  The set $In(Y, \alpha)$ is a closed 
  set. Moreover, in Fenchel-Nielsen 
  co-ordinates, we can write:
  \begin{align*}
    In(Y, \alpha) 
    = \{X\in \mc{T}(S) : \tau_{\alpha}(Y_{-t}^L) \le 
    \tau_{\alpha}(X)\le \tau_{\alpha}(Y_{-t}^R), 
    l_{\alpha}(X) = l_{\alpha}(Y)e^{-t} \textrm{for some $t$}\}
  \end{align*}
  Where $\tau_{\alpha}(Y_{-t}^R)$ 
  and $\tau_{\alpha}(Y_{-t}^L)$ 
  are functions of only 
  $l_{\alpha}(Y)$ and $t$.
\end{Lemma}

\begin{proof}
  Denote the second set by $A(Y)$.  Note that $Y_{-t}^R$ and
  $Y_{-t}^L$ are two smooth paths in $\mc{T}(S)$, which intersect at
  $Y$. In particular, they split $\mc{T}(S)$ into four complementary
  components (this is because stretch paths do not accumulate inside
  $\mc{T}(S)$), one of which is $A(Y)$.

  If $X \in A(Y)$, then consider $\mathrm{Stretch}(X,\lambda^L,t)$,
  and note that this path must $Y_{-t}^R$ at $X'$. 
  The contatenation of $\mathrm{Stretch}(X,\lambda^L,t)$ 
  from $X$ to $X'$, and the path $Y_{-t}^R$ from
  $X'$ to $Y$ shows that $X \in In(Y,\alpha)$.

  \begin{figure}
    \begin{center}
      \begin{tikzpicture}[scale = 2]
        \def \l{1}
        \def \tau{0}
      \end{tikzpicture}
    \end{center}
  \end{figure}

  Conversely, assume that $X \in In(Y,\alpha)$ but not in 
  $A(Y)$. Without loss of generality, 
  assume that $\tau_{\alpha}(X) > \tau_{\alpha}(Y_{-t}^R)$. 
  Let $\gamma(t)$ be the path from 
  $X$ to $Y$ which maximally stretches $\alpha$ along it. 
  By Theorem \ref{thm:conv-hull-env0}, it follows that 
  $\tau_{\alpha}(\gamma(t))
  \ge \tau_{\alpha}(\mathrm{Stretch}(X,\lambda^R,t))$ 
  for any $t$, and so $\gamma(t)$ cannot reach 
  $Y$ (by construction of $Y_{-t}^R$).
\end{proof}

\begin{Corollary}
  Let $S  =S_{1,1}$ or $S = S_{0,4}$. 
  For any $X,Y \in \mc{T}(S)$, we have:
  \begin{itemize}
    \item $Env(X,Y)$ is a unique geodesic from $X$ to $Y$ or,
    \item $Env(X,Y)$ is a combinatorial 
      quadrilateral whose edges are stretch paths
  \end{itemize}
\end{Corollary}

\begin{proof}
  If we are not in the first case, then 
  $\Lambda(X,Y)$ must be a simple closed curve, $\alpha$.
  Notice that $Env(X,Y) = In(Y,\alpha) \cap 
  \Outenv(X,\alpha)$. From Theorem \ref{thm:conv-hull-env0}, 
  we see that $\Outenv(\Lambda(X,Y))$ is a 
  cone starting at $X$ and bounded by 
  $\mathrm{Stretch}(X,\lambda^L,t)$ and 
  $\mathrm{Stretch}(X, \lambda^R,t)$. The intersection 
  of two cones in $\mc{T}(S)$ is a combinatorial 
  quadrilateral.
\end{proof}

This characterization allows us to prove
Lemma \ref{lem:uniform-bound-for-scc}
\begin{Theorem*}
  Let $X, Y \in \mc{T}(S)$ be such that 
  $\Lambda(X,Y) = \alpha$. For any $Z,Z' \in Env(X,Y)$, 
  if $l_{\alpha}(Z) = l_{\alpha}(Y)$, then 
  $d_{Th}(Z,Z')$ is bounded uniformly. In particular, 
  $w(X,Y)$ is uniformly bounded.
\end{Theorem*}

\begin{proof}
  We must have that 
  $\mathrm{Twist}_{\alpha}(Z,Z') \le 
  \mathrm{Twist}_{\alpha}(Y_{-t}^L,Y_{-t}^R)$, so 
  using the exact same 
  computations as in the previous subsection, 
  $d_{Th}(Z,Z') \le d_{Th}(Y_{-t}^L,Y_{-t}^R)$. 
  The theorem then follows from 
  Lemma \ref{lem:stretch-paths-uniform-bound}
\end{proof}

\bibliographystyle{plain}
\bibliography{main.bib}

\begin{thebibliography}{10}

\bibitem{BBFS09}
M.~Bestvina, K.~Bromberg, K.~Fujiwara, and J.~Souto.
\newblock Shearing coordinates and convexity of length functions on
  teichm\"{u}ller space.
\newblock {\em American Journal of Mathematics}, 135, 03 2009.

\bibitem{Bon01}
F.~Bonahon.
\newblock Shearing hyperbolic surfaces, bending pleated surfaces and thurston's
  symplectic form.
\newblock {\em Ann. Fac. Sci. Toulouse Math.}, 5, 07 2001.

\bibitem{BZ05}
F.~Bonahon and X.~Zhu.
\newblock The metric space of geodesic laminations on a surface ii: small
  surfaces.
\newblock pages 509--547, 05 2005.

\bibitem{DLRT20}
D.~Dumas, A.~Lenzhen, K.~Rafi, and J.~Tao.
\newblock Coarse and fine geometry of the thurston metric.
\newblock {\em Forum of Mathematics, Sigma}, 8:e28, 2020.

\bibitem{Gen15}
M.~Gendulphe.
\newblock {DERIVATIVES OF LENGTH FUNCTIONS AND SHEARING COORDINATES ON
  TEICHM{\"U}LLER SPACES}.
\newblock working paper or preprint, June 2015.

\bibitem{HOP21}
Y.~Huang, K.~Ohshika, and A.~Papadopoulos.
\newblock The infinitesimal and global thurston geometry of teichm{\"u}ller
  space, 2021.

\bibitem{LRT14}
A.~Lenzhen, K.~Rafi, and J.~Tao.
\newblock The shadow of a thurston geodesic to the curve graph.
\newblock {\em Journal of Topology}, 8, 05 2014.

\bibitem{Mar16}
B.~Martelli.
\newblock {\em An Introduction to Geometric Topology}.
\newblock Bruno Martelli, 2022.

\bibitem{Min93}
Y.~Minsky.
\newblock Teichm{\"u}ller geodesics and ends of hyperbolic 3-manifolds.
\newblock {\em Topology}, 32:625--647, 1993.

\bibitem{PW22}
Huiping Pan and Michael Wolf.
\newblock Ray structures on teichm\"uller space, 2022.

\bibitem{PT07}
A.~Papadopoulos and G.~Th\'{e}ret.
\newblock On teichmueller's metric and thurston's asymmetric metric on
  teichmueller apace. handbook of teichm\"{u}ller theory.
\newblock {\em European Mathematical Society Publishing House}, 1, 11, 2007.

\bibitem{HP92}
R.C. Penner and J.~Harer.
\newblock {\em Combinatorics of Train Tracks}.
\newblock Annals of Math. Studies 125, Princeton University Press, 1992.

\bibitem{Raf07}
K.~Rafi.
\newblock Thick-thin decomposition for quadratic differentials.
\newblock {\em Mathematical Research Letters}, 14:333--341, 2007.

\bibitem{Ther14}
G.~Th{\'e}ret.
\newblock Convexity of length functions and thurston's shear coordinates.
\newblock 2014.

\bibitem{Thu86}
W.~P. Thurston.
\newblock Minimal stretch maps between hyperbolic surfaces, preprint.
\newblock 1986.

\end{thebibliography}

\end{document}